\documentclass[12pt]{article}
\usepackage{mathrsfs}
\usepackage{amsmath}
\usepackage{float}
\usepackage{cite}
\usepackage{esint}
\usepackage[all]{xy}
\usepackage{amssymb}
\usepackage{amsfonts}
\usepackage{amsthm}
\usepackage{color}
\usepackage{graphicx}
\usepackage{hyperref}
\usepackage{bm}
\usepackage{indentfirst}
\usepackage{geometry}

\theoremstyle{plain}
\newtheorem{thm}{Theorem}[section]
\newtheorem{lem}[thm]{Lemma}
\newtheorem{prop}[thm]{Proposition}

\newtheorem{cor}[thm]{Corollary}
\theoremstyle{definition}
\newtheorem{defn}[thm]{Definition}
\theoremstyle{remark}
\newtheorem{remark}[thm]{Remark}


\numberwithin{equation}{section}

\begin{document}

\title{Compact harmonic RCD$(K,N)$ spaces are harmonic manifolds}
\author{\it Zhangkai Huang \thanks{Sun Yat-sen University: huangzhk27@mail.sysu.edu.cn}}
\date{\small\today}
\maketitle

\begin{abstract}
{}{In this paper, we study harmonic 
RCD$(K,N)$ spaces as the counterpart of harmonic
Riemannian manifolds \textit{with Ricci curvature bounded from below}.} We prove that a compact RCD$(K,N)$ space is isometric to a smooth closed Riemannian manifold if it satisfies either of the following harmonicity conditions:
\begin{enumerate}
\item[(1)] the heat kernel $\rho(x,y,t)$ depends only on the variable $t$ and the distance between points $x$ and $y$.
\item[(2)] the volume of the intersection of two geodesic balls depends only on their radii  and the distance between their centers. 
\end{enumerate}

\end{abstract}

\tableofcontents

\section{Introduction}

In $n$-dimensional Euclidean space $\mathbb{R}^n$, there exist harmonic functions that depend only on the geodesic distance {}{from the origin}. For example when $n>2$, the function $f(x)=|x|^{2-n}$ is harmonic on $\mathbb{R}^n\setminus \{0\}$. In the 1930s, {}{Ruse \cite{R31}} pioneered the study of radial harmonic functions on Riemannian manifolds, aiming to generalize the Euclidean case to more geometric settings.

{}{A function $f$ on a Riemannian manifold $(M^n,\mathrm{g})$ is called \textit{radial} for a point $p$ on a subset $U\subset M^n$ if there exists $h:[0,\infty)\to\mathbb{R}$ such that $f=h(\mathsf{d}_\mathrm{g}(\cdot,p))$ on $U$. Such functions are natural generalizations of radial functions on $\mathbb{R}^n$, and play a central role in harmonic analysis.}

{}{Since radial harmonic functions exist only under very specific conditions, Ruse introduced a definition to characterize these special classes of manifolds, as detailed below.}

\begin{defn}
A Riemannian manifold $(M^n,\mathrm{g})$ is said to be \textit{harmonic} {}{if, for each point $p$, there exists a normal coordinate neighborhood $\{x^1,\ldots,x^n\}_p
$ centered at $p$ such that the density function $\theta_p:=\sqrt{|\det (\mathrm{g}_{ij})|}$ is radial for $p$ on this neighborhood.}
\end{defn}

{As shown in \cite{B78,DR92,W50}, harmonic manifolds admit multiple equivalent characterizations.}
\begin{thm}\label{thm1.2}
A complete $n$-dimensional Riemannian manifold $(M^n,\mathrm{g})$ is  harmonic if and only if either of the following conditions holds.
\begin{enumerate}
\item[$(1)$] For any point $p\in M^n$, {}{there exists $r>0$ such that the function $\Delta (\mathsf{d}_\mathrm{g}(p,\cdot))^2$ is radial for $p$ on $B_r(p)$.}
\item[$(2)$] For any $p\in M^n$ there exists a non-constant radial harmonic function {}{for $p$} in a punctured neighborhood of $p$.
\item[$(3)$] Every small geodesic sphere in $M^n$ has constant mean curvature.
\item[$(4)$] {Every harmonic function satisfies the mean value property}.
\end{enumerate}
\end{thm}

When the space is simply connected, {}{additional equivalent characterizations hold}.

\begin{thm}[{}{{\cite[Thm. \!3]{CH11}}, {\cite[Thm.~1]{CH12}}}]\label{harm1}
{}{A simply connected and complete} Riemannian manifold is harmonic if and only if the volume of the intersection of two geodesic balls depends only on their radii  and the distance between their centers.
\end{thm}

\begin{thm}[{\cite[Thm.~1.1]{S90}}]\label{harm2}
{}{A simply connected and complete} Riemannian manifold $(M^n,\mathrm{g})$ is harmonic if and only if the heat kernel $\rho(x,y,t)$ depends only on the variable $t$ and the distance between points $x$ and $y${}{; that is,} it is of the form $\rho(x,y,t) = \rho(\mathsf{d}_{\mathrm{g}}(x,y),t)$.
\end{thm}

Given a Riemannian manifold $(M^n,\mathrm{g})$, let $H_i$ ($i=1,2$) be two functions on $M^n\times M^n$ such that the functions $H_i(x,\cdot)$ ($i=1,2$) are $L^2$-integrable for every $x\in M^n$. Then the convolution of $H_1$ and $H_2$, denoted as $H_1 \ast H_2$, is defined by
\[
H_1 \ast H_2(x, y) =\int_{M^n} H_1(x,z)H_2(y,z) \,\mathrm{d}\mathrm{vol}_{\mathrm{g}}(z).
\]
A function $H:M^n\times M^n\rightarrow \mathbb{R}$ is called radial kernel if $H(x, y)$ depends only on the geodesic distance between $x$ and $y$, that is, if $H =h\circ \mathsf{d}_{\mathrm{g}}$ for some $h:\mathbb{R}^+\rightarrow \mathbb{R}$.

 As an application of Thm. \ref{harm2}, through the approximation with characteristic functions, we immediately deduce the following theorem.

\begin{thm}[{\cite[Prop.~2.1]{S90}}]\label{harm3}
A simply connected and complete Riemannian manifold is harmonic if and only if the convolution
of radial kernel functions $H_1=h_1\circ \mathsf{d}_{\mathrm{g}}$ and $H_2=h_2\circ \mathsf{d}_{\mathrm{g}}$ {}{remains radial kernel, for any two compactly supported $L^2$-integrable functions $h_1,h_2$ on $\mathbb{R}$}.
\end{thm}

The primary objective of this paper is to characterize non-smooth contexts, namely $\mathrm{RCD}(K,N)$ metric measure spaces {}{that} are defined precisely in the following subsection and satisfy the conditions outlined in Thm. \ref{harm1}-\ref{harm3}.
\subsection{RCD spaces}

In this paper, by metric measure space, we mean a triple $(X, \mathsf{d},\mathfrak{m})$ such that  $(X, \mathsf{d})$ is a complete and  separable metric space and that $\mathfrak{m}$ is a non-negative Borel measure with full support on
$X$ and being finite on any bounded subset of $X$.

The curvature-dimension condition CD$(K,N)$, which involves a lower Ricci curvature bound $K\in\mathbb{R}$ and an upper dimension bound $N\in [1,\infty)$ for metric measure spaces in a synthetic sense, was first introduced in \cite{St06a,St06b} by Sturm and in \cite{LV09} by Lott-Villani respectively. Later, by adding a Riemannian structure into CD$(K,N)$ metric measure spaces, Ambrosio-Gigli-Savar\'{e} \cite{AGS14a}, Gigli \cite{G13,G15}, Erbar-Kuwada-Sturm \cite{EKS15} and
Ambrosio-Mondino-Savar\'{e} \cite{AMS19} introduced the notion of RCD$(K, N)$ spaces.  To be precise, an RCD$(K,N)$ space is a CD$(K,N)$ metric measure space with a Hilbertian $H^{1,2}$ Sobolev space.

{}{Although RCD$(K,N)$ spaces may lack smoothness, they preserve key analytical properties analogous to those in Riemannian manifolds with a lower Ricci bound $K$ and an upper dimension bound $N$.} {}{A cornerstone of the analysis is the heat kernel: Sturm \cite{St95,St96} and Jiang-Li-Zhang \cite{JLZ16}, established Gaussian estimates ensuring that the heat kernel admits a positive locally Lipschitz  continuous representation.} {Moreover, Bru\'e-Semola  \cite{BS20} revealed a deep structural property--every RCD$(K,N)$ space $({X},\mathsf{d},\mathfrak{m})$ bears a unique essential dimension $n\in [1,N]\cap\mathbb{N}$, denoted by $n:=\mathrm{dim}_{\mathsf{d},\mathfrak{m}}({X})$, such that the tangent space {}{at} $\mathfrak{m}$-almost every point is $\mathbb{R}^n$.}

{A pivotal advancement came with De Philippis-Gigli’s \cite{DG18} introduction of non-collapsed RCD$(K,N)$ spaces, which generalized non-collapsed limits of Riemannian manifolds with a lower Ricci curvature bound and a constant dimension. Here, the reference measure $\mathfrak{m}$ coincides with the $N$-dimensional Hausdorff measure $\mathscr{H}^N$, forcing $N$ to agree with the essential dimension of the space.}

{}{To further connect synthetic notions with classical regularity, recent research has focused on identifying conditions that enhance the regularity of RCD$(K,N)$ spaces. De Philippis-Gigli \cite{DG18} initiated this direction by introducing the concept of weakly non-collapsed RCD$(K,N)$ spaces, where $\mathfrak{m}$ is absolutely continuous with respect to $\mathscr{H}^N$. They conjectured that such spaces are equivalent to non-collapsed RCD$(K,N)$ spaces. This conjecture was later resolved affirmatively by Honda \cite{Hon20} and Brena-Gigli-Honda-Zhu \cite{BGHZ23}.} {In a parallel development, the author \cite{H23} proved that RCD$(K,N)$ spaces become non-collapsed under an isometrically heat kernel immersing condition. Moreover, for compact non-collapsed RCD$(K,n)$ spaces admitting eigenfunction-induced isometric immersions into Euclidean space, it was shown that they must be smooth.}

\subsection{Contributions}

In this paper, we focus on the following natural subclasses of RCD$(K,N)$ spaces, {}{motivated by the classical harmonic manifold theory.}

\begin{defn}[Strongly harmonic RCD$(K,N)$ space]
An RCD$(K,N)$ space $(X,\mathsf{d},\mathfrak{m})$ is said to be \textit{strongly harmonic} if its heat kernel $\rho$ depends only on distance and time. That is, there exists a function such that
{}{$\rho(x,y,t)=H(\mathsf{d}\left(x,y\right),t)$ for all $x,y\in X$ and $t>0$.}
\end{defn}

\begin{defn}[{}{Radially symmetric} RCD$(K,N)$ space]\label{defn1.7}
An RCD$(K,N)$ space $(X,\mathsf{d},\mathfrak{m})$ is said to be \textit{radially symmetric} if there exist non-constant eigenfunctions $\{\phi_i\}_{i=1}^m$ and a function $F:[0,\infty)\times [0,\infty)\rightarrow \mathbb{R}$ such that
\begin{equation}\label{eqn1.2}
\sum\limits_{i=1}^m \phi_i(x)\phi_i(y)=F(\mathsf{d}\left(x,y\right)),\ \forall x,y\in X.
\end{equation}
\end{defn}

The following theorems establish a hierarchy of regularity results for harmonic RCD$(K,N)$ spaces, culminating in smoothness under a volume homogeneity assumption.

\begin{thm}[Measure structure of strongly harmonic spaces]\label{mainthm2}
Let $(X,\mathsf{d},\mathfrak{m})$ be a strongly harmonic $\mathrm{RCD}(K,N)$ space. Then:
\begin{itemize}
\item[$(1)$] The measure $\mathfrak{m}$ is a constant multiple of the Hausdorff measure, i.e., $\mathfrak{m}=c\mathscr{H}^n$ for $c>0$, where $n=\mathrm{dim}_{\mathsf{d},\mathfrak{m}}(X)$. 
\item[$(2)$] {}{$(X,\mathsf{d},\mathscr{H}^n)$ is a non-collapsed $\mathrm{RCD}(K,n)$ space.}
\end{itemize}
\end{thm}
{}{This shows that strong harmonicity not only forces the measure to align with the Hausdorff measure (whose dimension equals the essential dimension), but also sharpens the RCD$(K,N)$ condition to the optimal RCD$(K,n)$ condition. These properties significantly simplify the analysis of such spaces.}

\begin{thm}[Smoothness of radially symmetric spaces]\label{mainthm1}
For a non-collapsed radially symmetric $\mathrm{RCD}(K,n)$ space $(X,\mathsf{d},\mathscr{H}^n)$, the metric space $(X,\mathsf{d})$ is isometric to an $n$-dimensional smooth closed Riemannian manifold $(M^n,\mathrm{g})$.
\end{thm}

{}{The radial symmetry condition (Defn. \ref{defn1.7}) thus provides a bridge between synthetic and classical geometry. Combining Thms. \ref{mainthm2} and \ref{mainthm1} yields:}

\begin{cor}\label{cor1.5}
Any compact strongly harmonic $\mathrm{RCD}(K,N)$ space $(X,\mathsf{d},\mathfrak{m})$ with $\mathrm{dim}_{\mathsf{d},\mathfrak{m}}(X)=n$ is isometric to an $n$-dimensional smooth closed Riemannian manifold.
\end{cor}

As a consequence of Cor. \ref{cor1.5}, we generalize Thm. \ref{harm1} to the 
RCD setting:

\begin{thm}[Geometric characterization of harmonicity]\label{thm1.11}
Let $(X,\mathsf{d},\mathfrak{m})$ be a compact $\mathrm{RCD}(K,N)$ space with $n=\mathrm{dim}_{\mathsf{d},\mathfrak{m}}(X)$. If the volume of the intersection of any two geodesic balls depends only on their radii and the distance between their centers, then:
\begin{enumerate}
\item[$(1)$] $\mathfrak{m}=c\mathscr{H}^n$ and $(X,\mathsf{d},\mathscr{H}^n)$ is non-collapsed.
\item[$(2)$] If $(X,\mathsf{d})$ is compact, then it is isometric to an $n$-dimensional smooth closed Riemannian manifold $(M^n,\mathrm{g})$.
\item[$(3)$] {}{If $(X,\mathsf{d})$ is compact and simply-connected, then it is a harmonic manifold in the classical sense. }
\end{enumerate}

\end{thm}
This fully answers the synthetic analogue of the classical question: when does volume homogeneity imply smoothness?
 \subsection{Outline of the proofs}
In Section \ref{sec3}, we begin by examining the smoothness of compact, non-collapsed, radially symmetric RCD$(K,n)$ spaces $(X,\mathsf{d},\mathscr{H}^n)$. Our main goal is to establish that the map $(\phi_1,\ldots,\phi_m):X\rightarrow \mathbb{R}^m$ constructed from eigenfunctions, meeting (\ref{eqn1.2}), realizes an isometric immersion. The smoothness then follows directly from \cite[Thm.~1.10]{H23}. Regarding the proof for the regularity of strongly harmonic RCD$(K,N)$ spaces $(X,\mathsf{d},\mathfrak{m})$, we employ a blow-up argument to demonstrate their non-collapsing nature and to refine the upper bound $N$ to its essential dimension $n:=\mathrm{dim}_{\mathsf{d},\mathfrak{m}}({X})$. This allows us to conclude the smoothness of compact, strongly harmonic RCD$(K,N)$ spaces by observing their radial symmetry.

In Section \ref{sec4}, we investigate the regularity of RCD$(K,N)$ spaces $(X,\mathsf{d},\mathfrak{m})$ where the volume of the intersection of two geodesic balls depends only on their radii and the distance between their centers. Using a blow-up argument again, we show that such spaces are non-collapsed and the dimension upper bound $N$ can be reduced to the essential dimension $n:=\mathrm{dim}_{\mathsf{d},\mathfrak{m}}({X})$. Furthermore, under the additional assumption of compactness for $X$, we analyze the smoothness of the space. A key observation is that any geodesic can be extended to a length equal to the diameter of $X$. This property enables us to construct bi-Lipschitz coordinate charts using distance functions, analogous to the bi-Lipschitz coordinates near regular points in Alexandrov spaces (see for instance \cite{BGP92}). By mapping back to Euclidean space, we derive the following pivotal equation for any two distinct points $x,\bar{x} \in X$.
\[
\lim_{r\rightarrow 0} \,\frac{1}{r}\left(\frac{\mathscr{H}^n\left(B_r(x)\setminus B_{\mathsf{d}(x,\bar{x})}(\bar{x})\right)}{\omega_n r^n}-\frac{1}{2}\right)=\Delta \mathsf{d}_{\bar{x}}(x).
\]
As shown in \cite{HT03}, in a smooth setting, the {}{left hand side of} this equation corresponds to the mean curvature of the geodesic sphere of radius $\mathsf{d}(x,\bar{x})$ centered at $\bar{x}$. Finally, the volume density function is found to be constant on geodesic spheres (i.e., radial with respect to the center), which implies that the space is strongly harmonic and, consequently, smooth.
\ 
\ 
\ 

\textbf{Acknowledgements} The author acknowledges the support of Fundamental Research Funds for the Central Universities, Sun Yat-sen University, Grant Number 24qnpy105. The author is indebted to Prof. Shouhei Honda and Prof. Huichun Zhang for their helpful discussions and insights. He also thanks Yuanlin Peng for his comments on this work.
\section{Preliminaries}

Throughout this paper, we use the following notation: 
\begin{itemize}
\item $C=C(k_1,\ldots,k_m)$ denotes a positive constant that depends only on the parameters $k_1,\ldots,k_m$.
\item $\Psi=\Psi(k_1,\ldots,k_m|l_1,\ldots,l_n)$ denotes a positive constant that depends only on $k_1,\ldots,k_m$, $l_1,\ldots,l_n$ with the property that $\Psi\to 0$ as $k_1,\ldots,k_m\to 0$ for any fixed $l_1,\ldots,l_n$.
\end{itemize}
For brevity, we may sometimes write $C$ or $\Psi$ without explicitly listing the parameters, noting that these constants may vary from line to line when no confusion arises. Now let $(X, \mathsf{d})$ be {a complete and separable metric space.}

\begin{itemize}
\item We denote by $C(X)$ and $\mathrm{Lip}(X,\mathsf{d})$ the set of continuous function and Lipschitz functions on $(X,\mathsf{d})$, respectively. The subset $\mathrm{Lip}_c(X,\mathsf{d})$ denotes compactly supported Lipschitz functions.
\item For any $f\in \mathrm{Lip}(X,\mathsf{d})$, the Lipschitz constant of $f$ is defined by
\[
\mathrm{Lip}\mathop{f}:=\sup_{x\neq y} \frac{|f(y)-f(x)|}{\mathsf{d}\left(y,x\right)};
\]
the pointwise Lipschitz constant of $f$ at $x\in X$ is defined by
\[
\mathrm{lip}\mathop{f}(x):=\limsup_{y\rightarrow x} \frac{|f(y)-f(x)|}{\mathsf{d}\left(y,x\right)}{}{\,(\text{or 0 if $x$ is isolated})}.
\]

\item We denote by $B_r^X(x)$ (or briefly $B_r({}{x})$) the open ball of radius $r>0$ centered at $x$, i.e. $\{y\in X:\mathsf{d}\left(x,y\right)<r\}$ and by $A_{r_1,r_2}^X(x)$ (or briefly $A_{r_1,r_2}(x)$) the annulus $\{y\in X:r_1<\mathsf{d}\left(x,y\right)<r_2\}$. The closure of $B_r^X(x)$ is denoted by $\bar{B}_r^X$ (or briefly $\bar{B}_r(x)$).
\item For any $n\in [1,\infty)$, we define $\mathscr{H}^n$ as the $n$-dimensional Hausdorff measure on $(X,\mathsf{d})$. Specifically, when $n\in\mathbb{N}$, we denote by $\mathscr{L}^n$ the standard Lebesgue measure on $\mathbb{R}^n$, which is consistent with the $n$-dimensional Hausdorff measure. Additionally, we let $\omega_n$ {}{denote the} Lebesgue measure of the unit ball $B_1(0_n)$ in $\mathbb{R}^n$. 
\item If $(X,\mathsf{d})$ is compact, then we define its diameter as:
\[
\mathrm{diam}(X,\mathsf{d}):=\sup_{x,y\in X}\mathsf{d}\left(x,y\right).
\]
\end{itemize}
\subsection{RCD space and heat kernel}


Let $(X,\mathsf{d},\mathfrak{m})$ be a metric measure space. 

The Cheeger energy $\mathrm{Ch}:L^2(\mathfrak{m})\rightarrow [0,\infty]$ is a convex, lower semi-continuous functional defined as
\[
\mathrm{Ch}(f):=\inf_{\{f_i\}}\left\{\int_X \left(\mathrm{lip}\mathop{f_i}\right)^2\mathrm{d}\mathfrak{m}\right\},
\]
where the infimum is taken among all sequences $\{f_i\}\subset \mathrm{Lip}(X,\mathsf{d})\cap L^2(\mathfrak{m})$ such that $\|f_i-f\|_{L^2(\mathfrak{m})}\rightarrow 0$. The Sobolev space $H^{1,2}(X,\mathsf{d},\mathfrak{m})$ is then defined as the set of $L^2$-integrable functions with finite Cheeger energy. We also define $H^{1,2}_{\text{loc}}(X,\mathsf{d},\mathfrak{m})$ as the space of functions belonging to $H^{1,2}(U,\mathsf{d},\mathfrak{m})$ for every bounded open $U \subset X$ and $L^p_{\text{loc}}(X,\mathfrak{m})$ analogously for $L^p$ spaces. When no confusion may arise, we use the abbreviations: $L^p:=L^p(X,\mathfrak{m})$ and $H^{1,2}:=H^{1,2}(X,\mathsf{d},\mathfrak{m})$.

For any $f\in H^{1,2}$, {}{Mazur’s Lem. yields a unique minimal relaxed slope $|\nabla f|\in L^2$ such that}
\[
\mathrm{Ch}(f)=\int_X {|\nabla f|}^2\, \mathrm{d}\mathfrak{m}.
\]
{}{This minimal relaxed slope satisfies locality, i.e.} $|\nabla f|=|\nabla h|$ $\mathfrak{m}$-a.e. on $\{f=h\}$. 

{}{A metric measure space is said to be \textit{infinitesimally Hilbertian} if the associated Sobolev space is a Hilbert spaces.} It is noteworthy that, as demonstrated in \cite{AGS14a,G15}, under this condition, for any two functions $f,h\in H^{1,2}$, the following function is well-defined:
\[
\langle\nabla f,\nabla h\rangle:=\lim_{\varepsilon\rightarrow 0}\frac{{|\nabla f+\varepsilon h|}^2-{|\nabla h|}^2}{2\varepsilon}\in L^1.
\]

Given an open subset $U\subset X$, $f \in H^{1,2}_{\text{loc}}(U, \mathsf{d},\mathfrak{m})$ is said to belong to the domain of the Laplacian, denoted  $f\in D(\mathbf{\Delta}, U)$,  if there exists a Radon measure $\mu$ on $U$ such that 

\[
\int_{U} \langle\nabla f,\nabla \psi\rangle \, \mathrm{d}\mathfrak{m} = -\int_{U} \psi \,\mathrm{d}\mu, \ \forall \psi \in \text{Lip}_c(U,\mathsf{d}) \cap L^1(U, |\mu|).
\]
In this case we write $\mathbf{\Delta} f\llcorner_{U}:= \mu$. If moreover $\mathbf{\Delta} f\llcorner_{U}\ll\mathfrak{m}\llcorner_{U}$ with density in $L^2_{\text{loc}}(U,\mathfrak{m})$, we denote by $\Delta f$ the unique function in $L^2_{\text{loc}}(U,\mathfrak{m})$ such that $\mathbf{\Delta}f = (\Delta f)\,\mathfrak{m}$ on $U$. For such functions, the following integration by parts formula holds for all $\psi \in H^{1,2}(U, \mathsf{d}, \mathfrak{m})$ with compact support:

\[
\int_U \langle \nabla f, \nabla \psi\rangle \, \mathrm{d}\mathfrak{m} = -\int_U \psi\Delta f  \, \mathrm{d}\mathfrak{m}.
\]
Finally, if \(\Delta f \in L^2(U, \mathfrak{m})\), we write \(f \in D(\Delta,U)\) (or simply $f\in D(\Delta)$ when $U=X$).

We are now in a position to introduce the definition of RCD$(K,N)$ spaces. See \cite{AGS15,AMS19,EKS15,G15} for details.
\begin{defn}
For $K\in \mathbb{R}$, $N\in (1,\infty)$, a metric measure space $(X,\mathsf{d},\mathfrak{m})$ is said to be an RCD$(K,N)$ space if it satisfies the following conditions.
\begin{enumerate}
\item[$(1)$] {}{It is infinitesimally Hilbertian.}
\item[$(2)$] There exists $x\in X$ and $C>1$ such that for any $r>0$ we have $\mathfrak{m}(B_r(x))\leqslant C\exp(Cr^2)$.
\item[$(3)$] Any $f\in H^{1,2}$ satisfying $|\nabla f|\leqslant 1$ $\mathfrak{m}$-a.e. has a 1-Lipschitz representative.
\item[$(4)$] For any $f\in D(\Delta)$ with $\Delta f\in H^{1,2}$ and any $\varphi \in \mathrm{Test}F\left({X},\mathsf{d},\mathfrak{m}\right)$ with
  $ \varphi \geqslant 0$, we have
\[
\frac{1}{2}\int_X |\nabla f|^2 \Delta \varphi\mathop{\mathrm{d}\mathfrak{m}}\geqslant \int_X \varphi \left(\frac{(\Delta f)^2}{N}+\langle \nabla f,\nabla \Delta f\rangle+K|\nabla f|^2 \right)\mathrm{d}\mathfrak{m},
\]
where $\mathrm{Test}F({X},\mathsf{d},\mathfrak{m})$ is the class of test functions defined by
 \[
\mathrm{Test}F({X},\mathsf{d},\mathfrak{m}):=\left\{\varphi\in \text{Lip}({X},\mathsf{d})\cap D(\Delta)\cap L^\infty:\Delta \varphi\in H^{1,2}\cap L^\infty\right\}.
\]
\end{enumerate}
If in addition $\mathfrak{m}=\mathscr{H}^N$, then $(X,\mathsf{d},\mathfrak{m})$ is said to be a non-collapsed RCD$(K,N)$ space.
\end{defn}

Throughout this paper, when referring to an RCD$(K,N)$ space, we always mean $N \in (1,\infty)$. For the remainder of this subsection, let us consider $(X, \mathsf{d}, \mathfrak{m})$ as a representative example of an RCD$(K,N)$ space.

By \cite[Thm.~1]{R12} and H\"{o}lder's inequality, a local $(1,2)$-Poincar\'{e} inequality holds: 

\[
\fint_{B_r(x)} \left|f - \fint_{B_r(x)} f \, \mathrm{d}\mathfrak{m} \right|\, \mathrm{d}\mathfrak{m} \leqslant 4re^{|K|r^2} \left(\fint_{B_{2r}(x)} |\nabla f|^2 \, \mathrm{d}\mathfrak{m}\right)^{1/2},
\]
for any $f \in H^{1,2}$ and any ball $B_r(x)\subset X$. This together with the following Bishop-Gromov volume growth inequality (see \cite[Thm. 5.31]{LV09}, \cite[Thm. 2.3]{St06b}) implies that RCD$(K,N)$ spaces are PI spaces.
\begin{thm}\label{BGineq}
For any $R>r>0$ $({}{\text{with}}\  R\leqslant \pi\sqrt{(N-1)/K}$ if $K>0)$, it holds 
\[
\dfrac{\mathfrak{m}\left(B_R(x)\right)}{\mathfrak{m}\left(B_r(x)\right)}\leqslant  \frac{V_{K,N}(R)}{ V_{K,N}(r)},
\]
where $V_{K,N} (r)$ denotes the volume of a ball of radius
$r$ in the $N$-dimensional model space with Ricci curvature $K$ defined as
\[ 
{}{V_{K,N}(t):=\left\{
\begin{array}{ll}
\int_0^t\sin^{N-1}\left(s\sqrt{K/(N-1)}\right)\mathrm{d}s, &\text{if}\  K>0,\\
t^{N},&\text{if}\  K=0,\\
 \int_0^t\sinh^{N-1}\left(s\sqrt{-K/(N-1)}\right)\mathrm{d}s, &\text{if}\  K<0.
 \end{array}
 \right.}
 \]  
\end{thm}


Building on the contributions by Sturm (see \cite[Prop. 2.3]{St95}, \cite[Cor. 3.3]{St96}) and Jiang-Li-Zhang (as detailed in Thm. \ref{JLZ} below), it is established that RCD$(K,N)$ spaces possess locally Lipschitz continuous heat kernels. More precisely, there exists a non-negative function $\rho$ on $X \times X \times (0,\infty)$ such that the unique solution to the heat equation can be expressed as follows.
\[
\mathrm{h}_t f=\int_X \rho(\cdot,y,t)f(y)\,\mathrm{d}\mathfrak{m}(y),\ \forall f\in L^2,\  \forall t>0,
\]
where by unique solution we mean $\{\mathrm{h}_t f\}_{t>0}$ solves the heat equation in $L^2$, i.e., \[
\frac{d }{d t} \mathrm{h}_t f=\Delta \mathrm{h}_t f,\ \text{in }L^2;\ \  \lim_{t\downarrow 0}\| \mathrm{h}_t f-f\|_{L^2}=0.
\]

\begin{thm}[Gaussian estimates for the heat kernel {\cite[{}{Thms. 1.1 and 1.2}]{JLZ16}}]\label{JLZ}
Let $\rho$ be the heat kernel of $(X,\mathsf{d},\mathfrak{m})$. Then given any $\varepsilon>0$, there exists {}{a constant $C=C(K,N,\varepsilon)$ such that
\begin{equation}\label{JLZineq}
C^{-1}\exp\left(-\frac{\mathsf{d}^2\left(x,y\right)}{(4-\varepsilon )t}-C t\right)\leqslant \mathfrak{m}(B_{\sqrt{t}}(x))\rho(x,y,t)\leqslant C\exp\left(-\frac{\mathsf{d}^2\left(x,y\right)}{(4+\varepsilon )t}+C t\right)
\end{equation}
holds for any $x,y\in X$ and
\begin{equation}\label{JLZineq2}
|\nabla_x \rho(x,y,t)|\leqslant \frac{C}{\sqrt{t}\mathop{\mathfrak{m}(B_{\sqrt{t}}(x))}}\exp\left(-\frac{\mathsf{d}\left(x,y\right)}{(4+\varepsilon)t}+C t\right)
\end{equation}
}
\end{thm}

\begin{remark}[Rescaled RCD spaces]
 For any
$a, b\in (0,\infty)$, the rescaled metric measure space $(X, a\mathsf{d}, b\mathfrak{m})$ is an RCD$(a^{-2}K, N)$ space, and its heat
kernel $\tilde{\rho}$ can be expressed as
\[
\begin{aligned}
\tilde{\rho}:X\times X\times (0,\infty)&\longrightarrow (0,\infty)\\
      (x,y,t)&\longmapsto b^{-1}\rho(x,y,a^{-2}t).
\end{aligned}
\]
\end{remark}
\subsection{Convergence of RCD spaces}

We omit the definition of pointed measured Gromov-Hausdorff (pmGH) convergence (see  \cite[Sec. 3]{GMS15} for details).
\begin{thm}[Precompactness of pointed RCD$(K,N)$ spaces under pmGH-convergence \cite{GMS15}]\label{GMS}
Let $\{(X_i,\mathsf{d}_i,\mathfrak{m}_i,x_i)\}$ be a sequence of pointed $\mathrm{RCD}(K,N)$ spaces such that
\[
0<\liminf_{i\rightarrow \infty}\mathfrak{m}_i(B_1^{X_i}(x_i))\leqslant \limsup_{i\rightarrow \infty}\mathfrak{m}_i(B_1^{X_i}(x_i))<\infty.
\]
Then this sequence has a subsequence $\{(X_{i(j)},\mathsf{d}_{i(j)},\mathfrak{m}_{i(j)},x_{i(j)})\}$ which $\mathrm{pmGH}$-converges to a pointed $\mathrm{RCD}(K,N)$ space $(X,\mathsf{d},\mathfrak{m},x)$.
\end{thm}

{}{Thus, we define the \textit{regular sets} as follows.}

\begin{defn}[Regular set]\label{regularset}
Let $(X,\mathsf{d},\mathfrak{m})$ be an RCD$(K,N)$ space. The tangent space at $x\in X$, denoted by $\mathrm{Tan}(X,\mathsf{d},\mathfrak{m},x)$, is defined as
\[
\left\{(Y,\mathsf{d}_Y,\mathfrak{m}_Y,y):\exists r_i\downarrow 0 \text{, s.t. }\left(X,{r_i}^{-1}\mathsf{d},(\mathfrak{m}(B_{r_i}(x)))^{-1}\mathfrak{m},x\right)\xrightarrow{\mathrm{pmGH}}(Y,\mathsf{d}_Y,\mathfrak{m}_Y,y)\right\}.
\]
The {}{set of $k$-dimensional regular points} is then defined as
\[
\mathcal{R}_k:=\left\{x\in X:\mathrm{Tan}(X,\mathsf{d},\mathfrak{m},x)=\left\{\left(\mathbb{R}^k,\mathsf{d}_{\mathbb{R}^k},{\omega_k}^{-1}\mathscr{L}^k,0_k\right)\right\}\right\}.
\]
\end{defn}
For the subsequent result concerning the existence of the \textit{essential dimension} in RCD spaces, we refer to \cite[Thm. 0.1]{BS20}.
\begin{thm}\label{BS}
Let $(X,\mathsf{d},\mathfrak{m})$ be an $\mathrm{RCD}(K,N)$ space. There exists a unique integer $n\in [1,N]$ such that $\mathfrak{m}\left(X\setminus \mathcal{R}_n\right)=0$. This integer $n$ is referred to as the essential dimension of $(X,\mathsf{d},\mathfrak{m})$ and is denoted by $\mathrm{dim}_{\mathsf{d},\mathfrak{m}}(X)$. In particular, $\mathfrak{m}$ can be represented as $\theta \mathscr{H}^n\llcorner \mathcal{R}_n$ for some Borel function $\theta$ defined on $X$.
\end{thm}
\begin{remark}[Density function {\cite[Thm.~4.1]{AHT18}}]\label{AHT}
Under Thm. \ref{BS}, {}{define \[
\begin{aligned}
\theta^\ast:X&\longrightarrow [0,\infty)\\
                    x&\longmapsto \left\{\begin{aligned}\lim_{r\downarrow 0} \frac{\mathfrak{m}(B_r(x))}{\omega_n r^n},&\ \text{if the limit exists,}\\
0,\ \ \ \ \ \ \ \ \ \ \ \ \ \ \ \  &\ \text{otherwise},
\end{aligned}\right.
\end{aligned}
\] and $\mathcal{R}_n^\ast:=\{x\in \mathcal{R}_n: \theta^\ast(x)>0\}.$}
Then $\mathfrak{m}\left(\mathcal{R}_n\setminus \mathcal{R}_n^\ast\right)=0$ and {}{$\theta^\ast=\theta$ $\mathfrak{m}\text{-a.e. }$on $\mathcal{R}_n^\ast$}.
\end{remark}

Specifically, in the case of non-collapsed RCD$(K,n)$ spaces, the following assertion regarding the density function and the essential dimension is valid.
\begin{thm}[{\cite[Cor.~1.7]{DG18}}]\label{DG18cor1.7} Let $(X,\mathsf{d},\mathscr{H}^n)$ be a non-collapsed
$\mathrm{RCD}(K, n)$ space. Then $\mathrm{dim}_{\mathsf{d},\mathscr{H}^n}(X)=n$ (notably $n$ is an integer) and {}{$\sup_X\theta^\ast\leqslant 1$}. Moreover, equality $\theta^\ast(x)=1$ occurs if and only if $x\in \mathcal{R}_n$.
\end{thm}

In the remainder of this subsection, we fix a sequence of pointed RCD$(K,N)$ spaces $\{(X_i, \mathsf{d}_i, \mathfrak{m}_i, x_i)\}$, which converges in the pointed measured Gromov-Hausdorff (pmGH) sense to another pointed RCD$(K,N)$ space $(X, \mathsf{d}, \mathfrak{m}, x)$.

We assume that readers are familiar with the definitions of $L^2$-weak and  $L^2$-strong convergence, and their counterparts for Sobolev functions, namely $H^{1,2}$-weak and $H^{1,2}$-strong convergence on various spaces. For reference, see \cite{AST17,AH17,GMS15} and \cite[Defn. 1.1]{H15}. We conclude this subsection by presenting some useful results related to this topic.

\begin{thm}[Arzel\`{a}-Ascoli Thm.]\label{AAthm}
Let $f_i\in C(X_i)$ $(i\in\mathbb{N})$ with {}{$
\sup_i |f_i(x_i)|<\infty$}. If for any $\varepsilon,R\in (0,\infty)$, there exists $\delta\in (0,1)$ such that for all $i\in \mathbb{N}$ it holds that
\[
|f_i(y_i)-f_i(z_i)|<\varepsilon, \ \forall y_i,z_i\in B_R(x_i) \text{ such that } \mathsf{d}_i(y_i,z_i)<\delta.
\]
Then after passing to a subsequence, there exists $f\in C(X)$ such that $\{f_i\}$ pointwisely converges to $f$ in the following sense:
\[
f_i(y_i)\rightarrow f(y) \text{\ whenever\ } X_i\ni y_i\rightarrow y\in X.
\]
\end{thm}
By a combination of the above theorem and Thm. \ref{JLZineq}, we have:
\begin{thm}[Pointwise convergence of heat kernels {\cite[Thm. 3.3]{AHT18}}]\label{pchk}
The heat kernels $\rho_i$ of $(X_i,\mathsf{d}_i,\mathfrak{m}_i)$ satisfy
\[
\lim_{i\rightarrow \infty} \rho_i(x_i,y_i,t_i)=\rho(x,y,t)
\]
whenever $X_i\times X_i\times (0,\infty)\ni (x_i,y_i,t_i)\rightarrow (x,y,t)\in X\times X\times (0,\infty)$.
\end{thm}
        {}{We also recall a precompactness theorem with respect to $L^2$-weak convergence (see \cite{GMS15,AH17, AST17}).}
\begin{thm}
Assume $f_i\in L^2(\mathfrak{m}_i)$ $(i\in\mathbb{N})$ such that $\sup_i \|f_i\|_{L^2(\mathfrak{m}_i)}<\infty$, then after passing to a subsequence $f_i$ $L^2$-weakly converges to some $f\in L^2(\mathfrak{m})$.
\end{thm}

{}{The following theorems, which will play a crucial role in our subsequent analysis, are adapted from \cite{AH17,AH18}.
\begin{thm}[From $L^1$-strong convergence to $L^2$ {\cite[Prop.~3.3]{AH17}}]\label{AH17}
Let $f_i\in L^1(\mathfrak{m}_i)$ $(i\in\mathbb{N})$ such that $\{\sigma\circ f_i\}$ $L^2$-strongly converges to $\sigma\circ f$ for some $f\in L^1(\mathfrak{m})$, where $\sigma:t\mapsto \mathrm{sign}(t)\sqrt{|t|}$ is the signed square root. If moreover $\sup_i \|f_i\|_{L^\infty(\mathfrak{m}_i)}<\infty$, then $\lim_{i\to\infty}\|f_i\|_{L^p(\mathfrak{m}_i)}=\|f\|_{L^p(\mathfrak{m})}$ for any $p\in [1,\infty)$.
\end{thm}}
 \begin{thm}[Stability of Laplacian on balls{\cite[Thm. 4.4]{AH18}}]\label{AH18}
 Let $f_i\in D(\Delta, B_R(x_i))$ $(i\in\mathbb{N})$ such that $\sup_i \|\Delta f_i\|_{L^2\left(B_R(x_i)\right)}+\sup_i \| f_i\|_{H^{1,2}\left(B_R(x_i)\right)}<\infty$ and {}{that} $\{f_i\}$ $L^2$-strongly converges to some $f\in L^2(B_R(x),\mathfrak{m})$ on $B_R(x)$. Then, for all $r\in (0,R)$ the following holds.
 \begin{enumerate}
 \item[$(1)$] $f|_{B_r(x)}\in D(\Delta,B_r(x))$.
 \item[$(2)$] $\{\Delta_i f_i\}$ $L^2$-weakly converges to $\Delta f$ on $B_r(x)$.
 \item[$(3)$] $\{f_i\}$ $H^{1,2}$-strongly converges to $f$ on $B_r(x)$.
 \end{enumerate}
 \end{thm}


\subsection{Calculus on RCD$(K,N)$ spaces}

{}{This subsection is aimed at presenting key results on calculus in} RCD$(K,N)$ spaces. For brevity, we omit definitions of $L^p$-one forms, $L^p$-tensor fields of type (0,2) over a Borel set $A\subset X$. These are denoted by $L^p(T^\ast(A,\mathsf{d},\mathfrak{m}))$ and $L^p((T^\ast)^{\otimes 2}(A,\mathsf{d},\mathfrak{m}))$, respectively. We also omit the definition of the pointwise {}{Hilbert-Schmidt norm, written as $|\cdot|_{\mathsf{HS}}$}, for $L^p$-tensor fields. Details can be found in \cite{G18b}. Let $(X,\mathsf{d},\mathfrak{m})$ be an RCD$(K,N)$ space.

\begin{thm}[Exterior derivative]
The linear operator $d$, called the exterior derivative, defined by
\[
\begin{aligned}
d:H^{1,2}&\longrightarrow L^2(T^\ast(X,\mathsf{d},\mathfrak{m}))\\
f&\longmapsto d\,f
\end{aligned}
\]
satisfies $|d\, f|=|\nabla f|$ $\mathfrak{m}$-a.e. for any $f\in H^{1,2}$. Moreover, the set $\{d\, f:f\in H^{1,2}\}$ is dense in $L^2(T^\ast(X,\mathsf{d},\mathfrak{m}))$.
\end{thm}

\begin{thm}[The Hessian]
For any $f\in \mathrm{Test}F\left({X},\mathsf{d},\mathfrak{m}\right)$, there exists a unique $T\in L^2\left((T^\ast)^{\otimes 2}({X},\mathsf{d},\mathfrak{m})\right)$, called the Hessian of $f$ and denoted by $ \mathop{\mathrm{Hess}}f$, such that for any $f_i\in \mathrm{Test}F\left({X},\mathsf{d},\mathfrak{m}\right)$ $(i=1,2)$,
\[
{}{2T(\nabla f_1,\nabla f_2)= \langle \nabla f_1,\nabla\langle \nabla f_2,\nabla f\rangle\rangle +\langle \nabla f_2,\nabla\langle \nabla f_1,\nabla f\rangle\rangle-\langle \nabla f,\nabla\langle \nabla f_1,\nabla f_2\rangle\rangle }
\]
holds for $\mathfrak{m}$-a.e.~$x\in {X}$.~{}{Moreover, for any $f\in \mathrm{Test}F\left({X},\mathsf{d},\mathfrak{m}\right)$ and any non-negative function $\varphi\in \mathrm{Test}F({X},\mathsf{d},\mathfrak{m})$, the following inequality holds.}

\begin{equation}\label{abc2.14}
\frac{1}{2}\int_{X}  \Delta \varphi \cdot {|\nabla f|}^2\,\mathrm{d}\mathfrak{m}\geqslant \int_{X}\varphi \left({|\mathrm{Hess}\,f|_{\mathsf{HS}}}^2+ \langle \nabla \Delta f,\nabla f\rangle+K{|\nabla f|}^2\right) \mathrm{d}\mathfrak{m}.
\end{equation}

\end{thm}

\begin{remark}\label{rmk2.16}
By the density of $\mathrm{Test}F({X},\mathsf{d},\mathfrak{m})$ in $D(\Delta)$, the Hessian $\mathop{\mathrm{Hess}}f$ is well-defined and belongs to $ L^2\left((T^\ast)^{\otimes 2}({X},\mathsf{d},\mathfrak{m})\right)$ for any $f\in D(\Delta)$. Furthermore, if $f_i\in D(\Delta)\cap \mathrm{Lip}(X,\mathsf{d})$ $(i=1,2)$, then we have $\langle \nabla f_1,\nabla f_2 \rangle\in H^{1,2}$ and $\mathrm{Hess}\, f_1=\mathrm{Hess}\, f_2$ $\mathfrak{m}$-a.e. on $\{f_1=f_2\}$. Thus for any $
\varphi\in H^{1,2}$, the following holds.
\begin{equation}\label{11eqn2.16}
\langle \nabla \varphi, \nabla \langle \nabla f_1,\nabla f_2 \rangle \rangle= \mathop{\mathrm{Hess}}f_1\left(\nabla f_2,\nabla\varphi\right)+ \mathop{\mathrm{Hess}}f_2\left(\nabla f_1,\nabla\varphi\right) \ \ \mathfrak{m}\text{-a.e.}
\end{equation}
Even for functions $f\in D(\Delta,B_{2r}(x))$, the Hessian on $B_{r}(x)$ can be interpreted as $\mathrm{Hess}\,(\varphi f)$, where $\varphi$ is a non-negative cut-off function in $\mathrm{Test}F(X,\mathsf{d},\mathfrak{m})$ such that $\varphi=1$ on $B_r(x)$, $\varphi=0$ outside $ B_{2r}(x)$, and that $r|\Delta \varphi|+r^2|\nabla \varphi|\leqslant C(K,N)$ $\mathfrak{m}$-a.e. on $B_{2r}(x)$. For details, see \cite[Thm. 6.7]{AMS14}, \cite{G18b} and \cite[Lem. 3.1]{MN14}.

\end{remark}
 
\begin{thm}[Canonical Riemannian metric \cite{AHPT21,GP22}]\label{thm2.17fff}
There exists a unique Riemannian metric $\mathrm{g}\in L^\infty((T^\ast)^{\otimes 2}(X,\mathsf{d},\mathfrak{m}))$ such that $|\mathrm{g}|_{\mathsf{HS}}=\sqrt{n}$ $\mathfrak{m}$-a.e. and that for any $f_1, f_2\in H^{1,2}
$ it holds
\[
\mathrm{g}(\nabla f_1,\nabla f_2)=\langle \nabla f_1,\nabla f_2\rangle,\ \mathfrak{m}\text{-a.e.}
\]
\end{thm}

We are now prepared to present a recent result in \cite[Thm. 1.10]{H23} to conclude this subsection. This theorem establishes that compact non-collapsed RCD$(K,n)$ spaces equipped with isometrically immersing eigenmaps are, in fact, smooth.

\begin{thm}\label{H22}
Let $(X,\mathsf{d},\mathscr{H}^n)$ be a compact non-collapsed $\mathrm{RCD}(K,n)$ space and {}{$\mathrm{g}$} be the canonical Riemannian metric on it. If there exist a finite number of non-constant eigenfunctions $\{\phi_i\}_{i=1}^m$ such that
\[
\mathrm{g}=\sum\limits_{i=1}^m d\,\phi_i\otimes d\,\phi_i ,
\]
then $(X,\mathsf{d})$ is isometric to an $n$-dimensional smooth closed Riemannian manifold $\left(M^n,\mathrm{g}\right)$.
\end{thm}

This theorem played an important role in proving the smoothness of compact isometrically heat kernel immersing RCD$(K,N)$ spaces as demonstrated in \cite{H23}. By ``isometrically heat kernel immersing" we mean that there exists a function $t\mapsto c(t)$ on $(0,\infty)$ such that for each $t > 0$, the function $x \mapsto (y \mapsto \sqrt{c(t)}\, \rho(x,y,t))$ realizes an isometric immersion into $L^2$.

\section{Symmetric and harmonic RCD$(K,N)$ spaces}\label{sec3}
This section is dedicated to proving Thms. \ref{mainthm2} and \ref{mainthm1}. The proof of Thm. \ref{mainthm1} relies on key estimates for eigenvalues and eigenfunctions, as well as the properties of the heat kernel, which are presented below.  
\begin{prop}[\cite{AHPT21,ZZ19}]\label{prop3.1}
Let $(X, \mathsf{d}, \mathfrak{m})$ be a compact $\mathrm{RCD}(K, N)$ space with $\mathfrak{m}(X)=1$. Let $0=\mu_0<\mu_1\leqslant \mu_2\leqslant \cdots \rightarrow \infty$ be all its eigenvalues counted with multiplicities and $\{\phi_i\}_{i=0}^\infty$ be the corresponding eigenfunctions, which form an $L^2$-orthonormal basis. Then, there exists a constant $C= C(K,N,\mathrm{diam}(X, \mathsf{d}))$, such that for all $i \geqslant 1$ we have
\[
\|\phi_i\|_{L^\infty}\leqslant C\,\mu_i^{\frac{N}{4}},\ \| \nabla \phi_i \|_{L^\infty}\leqslant C\,\mu_i^{\frac{N+2}{4}},\ C^{-1} \,i^{\frac{2}{N}}\leqslant \mu_i\leqslant C\,i^2.
\]
\end{prop}

\begin{remark}\label{rmk3.2}
{}{Under the assumptions of Prop. \ref{prop3.1}, by \cite{J14, JLZ16} and \cite[Appendix]{AHPT21}, the heat kernel $\rho$ of $(X,\mathsf{d},\mathfrak{m})$ admits the spectral representation
\[
\rho(x,y,t) = \sum_{i=0}^\infty \exp(-\mu_i t)\phi_i(x)\phi_i(y).
\]}
\end{remark}

\subsection{Smoothness of radially symmetric spaces}
Let $(X,\mathsf{d},\mathscr{H}^n)$ be a non-collapsed radially symmetric RCD$(K,n)$ space and let $\{\phi_i\}_{i=1}^m$ be non-constant eigenfunctions satisfying (\ref{eqn1.2}):
\begin{equation}\label{1.2c}
\sum_{i=1}^m \phi_i(x)\phi_i(y)=F(\mathsf{d}\,(x,y)),\, \forall x,y\in X.
\end{equation}
Here, $F:[0,\infty)\to\mathbb{R}$ is a fixed function, and $\mu_i$ denotes the eigenvalue corresponding to $\phi_i$. Let $C$ denote a constant depending on $K,m,n$,  $\mathrm{diam}({X},\mathsf{d}),\mathscr{H}^n({X})$, $\mu_1 \ldots,\mu_m$, $\left\|\phi_1\right\|_{L^2},\ldots,\left\|\phi_m\right\|_{L^2}.$

\begin{proof}[Proof of Thm. \ref{mainthm1}] 
Setting $x=y$ in (\ref{1.2c}) yields: 
\[
\sum\limits_{i=1}^m {\phi_i}^2(x)=F(0),\ \forall x\in X.
\]
From this, we obtain
\begin{equation}\label{eqn3.3}
\sum\limits_{i=1}^m {(\phi_i(x)-\phi_i(y))}^2=2\,\big(F(0)-F(\mathsf{d}\left(x,y\right))\big),\ \forall x,y\in X.
\end{equation}
Combining (\ref{eqn3.3}) with Prop. \ref{prop3.1}, we deduce
 \begin{equation}\label{eqn3.1}
\sup_{x\neq y} \left|\frac{F(0)-F(\mathsf{d}\left(x,y\right))}{{\mathsf{d}}^2\left(x,y\right)}\right|\leqslant C.
\end{equation}
 
 Fix a point
\[
x_0\in \mathcal{R}_n\cap \bigcap\limits_{i,j=1}^m \mathrm{Leb}(\langle\nabla \phi_i,\nabla\phi_j\rangle),
\]
where $\mathrm{Leb}(\langle\nabla \phi_i,\nabla\phi_j\rangle)$ denotes the set of Lebesgue points of $\langle\nabla \phi_i,\nabla\phi_j\rangle$.

 We claim
 \begin{equation}\label{eqn3.4}
 \liminf_{r\rightarrow 0} {r}^{-2}(F(0)-F(r))>0.
 \end{equation}

Assume the contrary, i.e. there exists $r_l\to 0$ such that ${r_l}^{-2}(F(0)-F(r_l))\to 0$ as $l\to \infty$. Since $x_0$ is a regular point, we have
 \[
 (X_l,\mathsf{d}_l,\mathscr{H}^n,x_0):=\left(X,{r_l}^{-1}\mathsf{d},\mathscr{H}^n,x_0\right)\xrightarrow{\mathrm{pmGH}} \left(\mathbb{R}^n,\mathsf{d}_{\mathbb{R}^n},\mathscr{L}^n,0_n\right).
 \]

For convenience, we denote the gradient and the Laplacian on $(X_l,\mathsf{d}_l,\mathscr{H}^n)$ by $\Delta_l,\nabla_l$ respectively. Additionally,  we use $B_r^l(x_0)$ to represent the ball $B_r^{X_l}(x_0)$. For each $i\in \{1,\cdots,m\}$ and $l\in \mathbb{N}$, let us set $\varphi_{i,l}:={r_l}^{-1}(\phi_i-\phi_i(x_0))$. Then we have
\[
|\nabla_l \,\varphi_{i,l}|=|\nabla \phi_i|,\  \Delta_l\,\varphi_{i,l} =r_l\,\Delta\, \phi_i=-r_l\, \mu_i \,\phi_i.
\]
Particularly, the forthcoming estimates are straightforwardly extracted from Prop. \ref{prop3.1}:
\begin{equation}\label{eqn3.5}
\left\|\nabla_l\, \varphi_{i,l}\right\|_{L^\infty(X_l)}\leqslant C;\ \ \left\|\Delta_l \,\varphi_{i,l}\right\|_{L^\infty(X_l)}\leqslant C\,r_l\rightarrow 0\ \ \text{as }l\rightarrow \infty.
\end{equation}

Recall that any harmonic function with bounded gradient on $\mathbb{R}^n$ must be linear. According to (\ref{eqn3.5}), Thms. \ref{AAthm} and \ref{AH18}, for each $i=1,\ldots,m$, the sequence $\{\varphi_{i,l}\}_l$ uniformly converges to a linear function $\varphi_i$ on {}{$B_2(0_n)\subset \mathbb{R}^n$}. Therefore, for each $i,j$, since $x_0\in  \mathrm{Leb}(\langle\nabla \phi_i,\nabla\phi_j\rangle)$, we have
\[
\begin{aligned}
\langle\nabla\phi_i,\nabla \phi_j\rangle(x_0)&=\lim_{r\downarrow 0}\fint_{B_r(x_0)}\langle\nabla \phi_i,\nabla \phi_j\rangle\mathop{\mathrm{d}\mathscr{H}^n}\\
\ &=\lim_{l\rightarrow \infty}\fint_{B^l_1(x_0)}\langle\nabla_l\, \varphi_{i,l},\nabla_l\, \varphi_{j,l}\rangle \mathop{\mathrm{d}\mathscr{H}^n}=\langle\nabla^{\mathbb{R}^n} \varphi_i,\nabla^{\mathbb{R}^n} \varphi_j\rangle.
\end{aligned}
\]

For any $\vec{y}\in \partial B_1(0_n)\subset \mathbb{R}^n$, by taking $y_l\in \partial B_{r_l}(x_0)$ such that $X_l\ni y_l\rightarrow \vec{y}$, we observe from (\ref{eqn3.3}) that
\[
\begin{aligned}
{}{2}\lim_{l\rightarrow \infty} \frac{F(0)-F\left(\mathsf{d}\left(x_0,y_l\right)\right)}{\mathsf{d}^2\left(x_0,y_l\right)}&=\lim_{l\rightarrow \infty}\left(\frac{r_l}{\mathsf{d}\left(x_0,y_l\right)}\right)^2\sum\limits_{i=1}^m \left(\frac{\phi_i(x_0)-\phi_i(y_l)}{r_l}\right)^2=\sum\limits_{i=1}^m {\varphi_i}^2(\vec{y})=0.
\end{aligned}
\]
As a result, $|\nabla^{\mathbb{R}^n} \varphi_i|\equiv|\nabla \phi_i|(x_0)=0$ $(i=1,\ldots,m)$.  This contradicts the non-degeneracy of $\phi_i$, as $x_0$ is arbitrary outside an $\mathscr{H}^n$-negligible set. Therefore we complete the proof of (\ref{eqn3.4}). 

Indeed, from the above discussion we know 
\begin{align}\label{faewoihfioeawhfoiaewhfoiea}
\sum_{i=1}^n {\varphi_i}^2(\vec{y})=c|\vec{y}|^2, \forall \vec{y}\in\mathbb{R}^n,\ \text{ where }c={2}\lim_{l\rightarrow \infty} \frac{F(0)-F\left(\mathsf{d}\left(x_0,y_l\right)\right)}{\mathsf{d}^2\left(x_0,y_l\right)}.
\end{align}
Taking Laplacian of both sides then implies $\sum_i {|\nabla^{\mathbb{R}^n} \varphi_i|}^2=\sum_i {|\nabla\phi_i|}^2(x_0)=nc$. As a result, we get
\begin{align}\label{faewiopfhjioaewhjfoipaehfoa}
{}{c}=2\,\lim_{r \downarrow 0} \frac{F(0)-F(r)}{r^2}={}{\sum\limits_{i=1}^m |\nabla \phi_i|^2(x)}, \, \mathscr{H}^n\text{-a.e.}\, x\in X.
\end{align}
\,

We assert: \begin{equation}\label{eqn3.8}
\left|\sum\limits_{i=1}^m d\,\phi_i\otimes d\,\phi_i-c\,\mathrm{g}\right|_\mathsf{HS}=0,\ \mathscr{H}^n\text{-a.e.},
\end{equation}
where $\mathrm{g}$ is the canonical Riemannian metric on $(X,\mathsf{d},\mathscr{H}^n)$. 

{}{Let $G:=\{x\in X:|\mathrm{g}|_\mathsf{HS}(x)=\sqrt{n}\}$, which satisfies $\mathscr{H}^n(X\setminus G)=0$ by Thm. \ref{thm2.17fff}.} Take
\[
x_0\in \bigcap\limits_{i,j=1}^m\mathrm{Leb}\big{(}(\langle\nabla \phi_i,\nabla \phi_j\rangle)^2\big{)}\cap \bigcap\limits_{i,j=1}^m \mathrm{Leb}(\langle\nabla \phi_i,\nabla\phi_j\rangle)\cap \mathcal{R}_n\cap G,
\] 
and use the notations as above. By (\ref{faewoihfioeawhfoiaewhfoiea}) and (\ref{faewiopfhjioaewhjfoipaehfoa}), $c\,\mathrm{g}_{\mathbb{R}^n}=\sum_i d\,\varphi_i\otimes d\,\varphi_i$. The $H^{1,2}$-strong convergence of $\{\varphi_{i,l}\}_l$ for each $i$ on {}{$B_2(0_n)\subset\mathbb{R}^n$} (Thm. \ref{AH18}) as well as (\ref{eqn3.5}) and Thm. \ref{AH17} implies that{}{
\[
\begin{aligned}
&\left|c\,\mathrm{g}-\sum\limits_{i=1}^m d\,\phi_i\otimes d\,\phi_i\right|_\mathsf{HS}^2(x_0)=nc^2+\sum_{i,j}(\langle\nabla\phi_i,\nabla \phi_j\rangle(x_0))^2-2c\sum_i |\nabla \phi_i|^2(x_0)\\
=\ &nc^2+\lim_{r \rightarrow 0}\fint_{B_r(x_0)}\left(\sum_{i,j}(\langle\nabla\phi_i,\nabla \phi_j\rangle)^2-2c\sum_i |\nabla \phi_i|^2\right)\mathrm{d}\mathscr{H}^n\\
=\ &nc^2+\lim_{l \rightarrow \infty}\fint_{B_1^l(x_0)}\left(\sum_{i,j}(\langle\nabla_l\varphi_{i,l},\nabla_l \varphi_{j,l}\rangle)^2-2c\sum_i |\nabla_l \varphi_{i,l}|^2\right)\mathrm{d}\mathscr{H}^n\\
=\ &nc^2+\fint_{B_1(0_n)}\left(\sum_{i,j}(\langle\nabla^{\mathbb{R}^n}\varphi_{i},\nabla^{\mathbb{R}^n} \varphi_{j}\rangle)^2-2c\sum_i |\nabla^{\mathbb{R}^n} \varphi_{i}|^2\right)\mathrm{d}\mathscr{L}^n\\
=\ &\fint_{B_1(0_n)}\left|c\,\mathrm{g}_{\mathbb{R}^n}-\sum\limits_{i=1}^m d\,\varphi_i\otimes d\,\varphi_i\right|_\mathsf{HS}^2\mathrm{d}\mathscr{L}^n=0.
\end{aligned}
\]}As $x_0$ is taken to be arbitrary {}{outside a negligible set}, we have completed the proof of (\ref{eqn3.8}). Therefore, the map $(\phi_1,\ldots,\phi_m):X{}{\rightarrow} \mathbb{R}^m$ is an isometric immersion. Since $(X,\mathsf{d},\mathscr{H}^n)$ is compact and non-collapsed, it follows from Thm. \ref{H22} that $(X,\mathsf{d})$ is isometric to a smooth closed Riemannian manifold.
\end{proof}

\begin{remark}[Local radial symmetry]
The radial symmetry condition can be relaxed to a local version: if for every $x\in X$ there exists $\epsilon_x>0$ such that $\sum_{i=1}^m \phi_i(x)\phi_i(y) = F\left(\mathsf{d}\left(x, y\right)\right)$ holds for any $y\in B_{\varepsilon_x}(x)$, the same conclusion holds.

\end{remark}

\subsection{Regularity of strongly harmonic spaces}
We now examine strongly harmonic RCD$(K,N)$ spaces. Consider an RCD$(K,N)$ space $(X,\mathsf{d},\mathfrak{m})$ with essential dimension $\mathrm{dim}_{\mathsf{d},\mathfrak{m}}(X)=n$. The space is called strongly harmonic if its heat kernel $\rho$ satisfies
\begin{equation}\label{eeqn4.1}
\rho(x,y,t)=H\left(\mathsf{d}\left(x,y\right),t\right),\ \forall x,y\in X,\ \forall t>0.
\end{equation}
for some function $H:[0,\infty)\times [0,\infty)\rightarrow \mathbb{R}$. 
We begin with the proof of Thm. \ref{mainthm2}, which parallels the proof of \cite[Thm. 1.7]{H23}.
\begin{proof}[Proof of Thm. \ref{mainthm2}]We first prove the following key limit formula: there exists $c>0$ such that for any $\sigma\geqslant 0$ it holds
\begin{equation}\label{eqn4.1}
\lim_{t\downarrow 0} t^n H(\sigma t,t^2)=c^{-1}{\omega_n}^{-1}\mathop{(4\pi )^{-\frac{n}{2}}} \exp\left(-\frac{\sigma^2}{4}\right).
\end{equation}

Fix $x_0\in \mathcal{R}_n^\ast$ and take $r_i\to 0$. Then we have
\[
(X_i,\mathsf{d}_i,{}{\mathfrak{m}_i},x_0):=\left(X,{r_i}^{-1}{}{\mathsf{d}},\left(\mathfrak{m}(B_{r_i}(x_0))\right)^{-1}\mathfrak{m},x_0\right)\xrightarrow{\mathrm{pmGH}} \left(\mathbb{R}^n,\mathsf{d}_{\mathbb{R}^n},{\omega_n}^{-1}\mathscr{L}^n,0_n\right).
\]

The heat kernel $\rho_i$ of each $(X_i,\mathsf{d}_i,\mathfrak{m}_i)$ satisfies
\begin{equation}\label{eeeqn4.2}
\rho_i(x_i,y_i,1)=\mathfrak{m}(B_{r_i}(x_0))\,\rho(x_i,y_i,{r_i}^2), \ \forall x_i,y_i\in X_i.
\end{equation}

Assume first $\sigma>0$. For each sufficiently large $i$, we choose $x_i,y_i\in B_{2\sigma r_i}(x_0)$ such that {}{$\mathsf{d}\left(x_i,y_i\right)=\sigma r_i$} and that $X_i\ni x_i\to \vec{x}\in \mathbb{R}^n$, $X_i\ni y_i\to \vec{y}\in \mathbb{R}^n$, respectively. By Thm. \ref{pchk},
\[
\lim_{i\rightarrow \infty}\rho_i(x_i,y_i,1)=(4\pi )^{-\frac{n}{2}}\exp\left(-\frac{|\vec{x}-\vec{y}|^2}{4}\right)=(4\pi )^{-\frac{n}{2}}\exp\left(-\frac{\sigma^2}{4}\right).
\]
Combining this with (\ref{eeqn4.1}) and (\ref{eeeqn4.2}), we derive
{}{\begin{equation}\label{eqnaaa5.13}
\begin{aligned}
&\omega_n\, \theta^\ast(x_0)\lim_{i\rightarrow \infty} {r_i}^{n} H(\sigma r_i,{r_i}^2)=\lim_{i\rightarrow \infty}\frac{\mathfrak{m}(B_{r_i}(x_0))}{{r_i}^n}{r_i}^n H\left(\sigma r_i,{r_i}^2\right)=\lim_{i\rightarrow \infty}\mathfrak{m}(B_{r_i}(x_0))H\left(\sigma r_i,{r_i}^2\right)\\
&=\lim_{i\to\infty}\mathfrak{m}(B_{r_i}(x_0))\rho(x_i,y_i,r_i^2)=\lim_{i\to\infty}\rho_i(x_i,y_i,1)=(4\pi )^{-\frac{n}{2}}\exp\left(-\frac{\sigma^2}{4}\right),
\end{aligned}
\end{equation}}{where $\theta^\ast$ is defined in Rmk. \ref{AHT}.} Since this holds for any $r_i\to 0$ and any $x_0\in \mathcal{R}_n^\ast$, (\ref{eqn4.1}) follows. {}{Moreover, we have $\theta^\ast|_{\mathcal{R}_n^\ast}\equiv c$ for some $c>0$. The case $\sigma=0$ can be handled analogously by repeating the same argument. }

{}{We now verify
\begin{equation}\label{eeeeeeeeeeqn3.14}
\inf_{r\in (0,1),\, x\in X}r^{-n}\,\mathfrak{m}(B_r(x))>0.
\end{equation}}

We proceed by contradiction. {}{By Thm. \ref{BGineq}, we may} assume there exists $r_i\to 0$ and a sequence of points $\{x_i\}\subset X$ such that $\lim_{i\to\infty}{r_i}^{-n}\,\mathfrak{m}(B_{r_i}(x_i))=0$. After passing to a subsequence, we have 
\[
(X_i,\mathsf{d}_i,{}{\mathfrak{m}_i},x_0):=\left(X,{r_i}^{-1}{}{\mathsf{d}},\left(\mathfrak{m}(B_{r_i}(x_0))\right)^{-1}\mathfrak{m},x_i\right)\xrightarrow{\mathrm{pmGH}} \left(X_\infty,\mathsf{d}_\infty,\mathfrak{m}_\infty,x_\infty\right),
\]
for some pointed RCD$(0,N)$ space $\left(X_\infty,\mathsf{d}_\infty,\mathfrak{m}_\infty,x_\infty\right)$. The heat kernel $\rho_\infty$ of this space satisfies
\[
\begin{aligned}
\rho_\infty(x_\infty,x_\infty,1)&=\lim_{i\rightarrow \infty}\mathfrak{m}(B_{r_i}(x_i))\mathop{\rho(x_i,x_i,r_i^2)}=\lim_{i\rightarrow \infty}\mathfrak{m}(B_{r_i}(x_i))\mathop{H(0,{r_i}^2).}
\end{aligned}
\]

Combining (\ref{eqn4.1}) then yields $\rho_\infty(x_\infty,x_\infty,1)=0$, which contradicts (\ref{JLZineq}). Thus  (\ref{eeeeeeeeeeqn3.14}) follows, which along with \cite[Thm.~2.4.3]{AT04} implies $\mathscr{H}^n \ll \mathfrak{m}$. By Rmk.~\ref{AHT}, $\mathfrak{m}=c\mathscr{H}^n$. Finally, applying (\ref{eeeeeeeeeeqn3.14}) and \cite[Thms.~1.5 and 2.22]{BGHZ23}, we conclude that $(X,\mathsf{d},\mathscr{H}^n)$ is a non-collapsed RCD$(K,n)$ space.

\end{proof}

Assume in addition $(X,\mathsf{d},\mathscr{H}^n)$ is compact with $\mathscr{H}^n(X)=1$. Let {}{$C$} be the constant from Prop.~\ref{prop3.1}.  We show that the space is radially symmetric (and thus smooth by Thm.~\ref{mainthm1})
\begin{proof}[Proof of Cor.~\ref{cor1.5} (\textbf{Smoothness})]
Without loss of generality, we still assume $\mathscr{H}^n(X)=1$. By Rmk. \ref{rmk3.2}, the heat kernel admits the expansion:
\[
\sum\limits_{i=0}^\infty \exp(-\mu_i t)\phi_i(x)\phi_i(y)=\rho(x,y,t)=H\left(\mathsf{d}\left(x,y\right),t\right),\, \forall x,y\in X, \, \forall t>0.
\]

Let $\{\phi_1,\ldots,\phi_m\}$ be an $L^2$-orthonormal basis for the $\mu_1$-eigenspace. For any $x,y\in X$ and any $t>0$,
\begin{equation}\label{eeqn4.3}
\sum\limits_{i=1}^m\phi_i(x)\phi_i(y)=\exp(\mu_1 t)(H(\mathsf{d}(x,y),t)-1)-\sum\limits_{i=m+1}^\infty \exp((\mu_1-\mu_i) t)\,\phi_i(x)\phi_i(y).
\end{equation}

Let $N_0$ be the integer such that for all $i\geqslant N_0$ it holds $\mu_i\geqslant C^{-1} \mathop{i^{2/n}}\geqslant 2\mu_1$. The second term of the right hand side of (\ref{eeqn4.3}) satisfies
\begin{equation}\label{eeqn4.4}
\begin{aligned}
&\left|\sum\limits_{i=m+1}^\infty \exp((\mu_1-\mu_i)t)\,\phi_i(x)\phi_i(y)\right|
\leqslant C \sum\limits_{i=m+1}^\infty \exp((\mu_1-\mu_i )t)\, {}{{i}^n}\\
=\,&{}{C}\sum\limits_{i=m+1}^{N_0} \exp((\mu_1-\mu_i) t)\, {}{{i}^n}+{}{C}\sum\limits_{i=N_0+1}^{\infty} \exp((\mu_1-\mu_i) t)\, {}{{i}^n}\\
\leqslant\, &{}{C}\sum\limits_{i=m+1}^{N_0} \exp((\mu_1-\mu_i) t)\, {}{{i}^n}+{}{C}\sum\limits_{i=N_0+1}^{\infty} \exp\left(-\frac{i^{2/n}t}{2C}\right)\, {}{{i}^n}.\\
\end{aligned}
\end{equation}
{}{Since $\lim_{s\to \infty}\exp\left(-\frac{s^{2/n}}{4C}\right)\mathop{s^{n}}=0$, there exists $T\gg 1$ such that for any $t>T$ and $i\geqslant N_0$ we have $\exp\left(-\frac{i^{2/n}t}{4C}\right)\, {i}^n\leqslant 1$, which implies}{}{\begin{equation}\label{eeeqn3.18}
\begin{aligned}
&\lim_{t\to\infty}\sum\limits_{i=N_0+1}^{\infty} \exp\left(-\frac{i^{2/n}t}{2C}\right)\, i^n\leqslant  \lim_{t\to\infty}\sum\limits_{i=N_0+1}^{\infty} \exp\left(-\frac{i^{2/n}t}{4C}\right)\\
\leqslant\ & \lim_{t\to\infty}\int_{0}^\infty \exp\left(-\frac{s^{2/n} t}{4C}\right)\,\mathrm{d}s=\lim_{t\to\infty}t^{-2/n}\int_{0}^\infty \exp\left(-\frac{s^{2/n}}{4C}\right)\,\mathrm{d}s=0.
\end{aligned}
\end{equation}}

From (\ref{eeqn4.3}), (\ref{eeqn4.4}) and (\ref{eeeqn3.18}), we conclude
\[
\sum\limits_{i=1}^m\phi_i(x)\phi_i(y)=\lim_{t\rightarrow \infty} \exp(\mu_1 t)(H\left(\mathsf{d}\left(x,y\right),t\right)-1):=F\left(\mathsf{d}\left(x,y\right)\right),
\]
proving radial symmetry. Now the smoothness of $(X,\mathsf{d})$ follows from Thm. \ref{mainthm1}.

\end{proof}

\section{Regularity of RCD spaces with volume homogeneity}\label{sec4}
This section is devoted to the proof of Thm. \ref{thm1.11}, which establishes smoothness for compact RCD spaces under the volume homogeneity assumption. We begin by introducing
key tools, including the co-area formula and measure disintegration.

\subsection{Co-area formula and disintegration formula on RCD spaces}\label{sec4.1}

We first recall the co-area formula, a fundamental tool in geometric analysis. For further details, see \cite[Prop. 4.2]{M03} (for the general statement) and \cite{ABS19, BPS23} (for applications to perimeter measures in RCD spaces). Additionally, \cite[Rmk. 3.5]{GH16} clarifies the relationship between the total variation of a Lipschitz function and its weak upper gradient.

\begin{thm}[Co-area formula]\label{coareafor}Let $(X, \mathsf{d}, \mathscr{H}^n)$ be a compact non-collapsed $\mathrm{RCD}(K, n)$ space and let $v$ be a Lipschitz function on $X$. Then:
\begin{itemize} 
\item[$(1)$] The set $\{v > t\}$ has finite perimeter for $\mathscr{L}^1$-a.e. {}{$t\in\mathbb{R}$}.
\item[$(2)$] For every Borel function $f: X \rightarrow [0, \infty]$, we have
\[
\int_X f \,|\nabla v| \, \mathrm{d}\mathscr{H}^n = \int_{-\infty}^\infty \left( \int_{\{v=t\}} f \, \mathrm{d}\mathscr{H}^{n-1} \right) \mathrm{d}t.
\]
\end{itemize}
\end{thm}

Next, we focus on disintegration with respect to the distance function $\mathsf{d}_x:=\mathsf{d}\left(x,\cdot\right)$ for a fixed point $x$ in a compact RCD$(K,N)$ space $(X,\mathsf{d},\mathfrak{m})$. See for instance \cite[Sec. 3]{CM20}.

We recall a geometric property of RCD spaces.
\begin{thm}[{\cite[Thm.~1.3]{D20}}]\label{RCDnb}
Any $\mathrm{RCD}(K,N)$ space $(X,\mathsf{d},\mathfrak{m})$ admits no branching geodesics. Especially, there exist no two distinct unit-speed geodesics $\gamma_1,\gamma_2:[0,1]\to X$ such that 
\[
\gamma_1(s)=\gamma_2(s),\ \  \forall s\in [0,t], \ \ \text{but}\ \ \gamma_1(t)\neq\gamma_2(t)\ \ \text{for some}\ \  t\in (0,1). 
\]    
\end{thm}

Let us fix a point $x$ in a compact RCD$(K,N)$ space $(X,\mathsf{d},\mathfrak{m})$ with $\mathrm{dim}_{\mathsf{d},\mathfrak{m}}(X)\geqslant 2$. The construction begins with the set:
\[
\Gamma:=\{(y,z)\in X\times X:\mathsf{d}_x(y)-\mathsf{d}_x(z)=\mathsf{d}\left(y,z\right)\},
\]
where the transpose is defined as $\Gamma^{-1}:=\{(y,z)\in X\times X: (z,y)\in \Gamma\}$. This thereby induces the transport relation $R:=\Gamma\cup \Gamma^{-1}$. For any $y\in X$, we define the section $\Gamma(y)=\{z\in X:(y,z)\in \Gamma\}$ through the first coordinate (with $R(y)$ defined analogously via symmetry in both coordinates). Building on \cite{C14}, we then construct the forward branching set:
\[
{\mathcal{E}=\{w\in X:\exists\  (y,z)\in \Gamma(w)\times \Gamma(w) \setminus  R\}. }
\]
Under the assumption $\mathrm{dim}_{\mathsf{d},\mathfrak{m}}(X)\geqslant 2$ and by Thm. \ref{RCDnb}, the backward branching point set 
\[
{}{\{w\in X:\exists\  (y,z)\in \Gamma^{-1}(w)\times \Gamma^{-1}(w) \setminus  R\} }
\]
reduces exactly to the singleton $\{x\}$.

The non-branched transport set $\mathcal{T}:=X\setminus (\mathcal{E}\cup\{x\})$ is a Borel set such that $\mathfrak{m}(X\setminus \mathcal{T})=0$ by \cite[Thm. 5.5]{C14}. According to \cite[Thm. 4.6]{C14} (see also \cite{BC13}), the non-branched transport relation $\mathcal{R}:=R\cap (\mathcal{T}\times \mathcal{T})$ becomes an equivalence relation on $\mathcal{T}$, where each section $R(y)\subset (X,\mathsf{d})$ for $y\in \mathcal{T}$ forms a unit speed geodesic (parameterized by a closed interval) starting from $x$.

This equivalence relation partitions  $\mathcal{T}$ into a disjoint family of equivalence classes $\{X_\alpha\}_{\alpha\in Q}$ (where $Q$ is an index set), with the quotient map $\mathfrak{Q}:\mathcal{T}\rightarrow Q$ uniquely determined by the correspondence  
\[
\alpha=\mathfrak{Q}(y) \iff y\in X_\alpha.
\]
We equip $Q$ with the quotient $\sigma$-algebra $\mathscr{Q}$, derived from the $\sigma$-algebra $\mathscr{X}$ over $X$ consisting of $\mathfrak{m}$-measurable sets. This quotient $\sigma$-algebra is defined by
\[
A\in \mathscr{Q}\iff \mathfrak{Q}^{-1}(A)\in \mathscr{X},
\]
making it the finest $
\sigma$-algebra for which $\mathfrak{Q}$ is measurable. The normalized quotient measure is then defined by $\mathfrak{q}:=\mathfrak{m}(X)^{-1}\mathfrak{Q}_\sharp (\mathfrak{m})$, meaning that $\mathfrak{q}$ is obtained by pushing forward the measure $\mathfrak{m}(X)^{-1}\mathfrak{m}$, with $\mathfrak{Q}_\sharp \mathfrak{m}\ll \mathfrak{q}$ following immediately. 

Indeed, the quotient map $\mathfrak{Q}$ admits an explicit construction that embeds $Q$ into $X$ while maintaining $\mathcal{A}$-measurability, where $\mathcal{A}$ denotes the $\sigma$-algebra generated by analytic sets of $X$ (see \cite{CM17} and \cite[Lem. 3.8]{CM201}).

Combining these results with the classical disintegration theorem \cite[Prop. 452F]{F03} and recent localization results (\cite[Thm. 9.5]{BC13} and \cite[Thm. 5.1]{CM17b}), we obtain:


\begin{thm}[Needle decomposition]\label{thmdisinte}
For the measure $\mathfrak{m}\llcorner_\mathcal{T}$, there exists a disintegration:
\begin{equation}\label{disintefor}
\mathfrak{m}\llcorner_\mathcal{T}:=\int_Q \mathfrak{m}_\beta \mathop{\mathrm{d}\mathfrak{q}(
\beta)},
\end{equation}
where for $\mathfrak{q}$-almost everywhere $\beta\in Q$ the following holds: 
\begin{enumerate}
\item[$(1)$] Each needle $X_\beta$ is a geodesic starting from $x$, equipped with a measure $h_\beta\mathscr{H}^1$, where $h_\beta$ a $\mathrm{CD}(K,N)$ density on $X_\beta$ $($i.e., $\mathfrak{m}_\beta=h_\beta \mathscr{H}^1\llcorner_{X_\beta})$.
\item[$(2)$] The map $\beta\mapsto \mathfrak{m}_\beta(B)$ is $\mathfrak{q}$-measurable for any Borel set $B\subset X$ (This ensures the disintegration is well-defined).
\end{enumerate}
{}{As a result of (\ref{disintefor}) and our construction}, the following identity holds for any $\mathfrak{m}$-measurable set $B$ and $\mathfrak{q}$-measurable set $A$:
\[
\mathfrak{m}(B\cap \mathfrak{Q}^{-1}(A))=\int_A \int_{X_\beta} \chi_B\,{}{h_\beta}\mathop{\mathrm{d}\mathscr{H}^1}\mathop{\mathrm{d}\mathfrak{q}(\beta).}
\]

\end{thm}
 
\begin{remark}\label{rmk4.6aaaaaaaaa}
By stating $h_\beta$ a CD$(K,N)$ density, we mean $\left(\overline{X}_\beta,\mathsf{d},h_\beta\mathscr{H}^1\right)$ verifies the $\mathrm{CD}(K,N)$ condition, where $\overline{X}_\beta$ is the closure of $X_\beta$. While we omit the definition of CD$(K,N)$ metric measure spaces (see \cite{LV09,St06a,St06b}), we highlight some key properties for such densities $h_\beta$ as follows.
\begin{enumerate}
\item[$(1)$] The function $h_\beta$ is locally semi-concave in the interior of $X_\beta$. That is, for every $z_0$ in the interior of $X_\beta$, there exists a constant $C(z_0,\beta)\in \mathbb{R}$ such that the function $z\mapsto \left(h_\beta-C(z_0,\beta)z^2\right)$ is concave in a neighborhood (in $X_\beta$) of $z_0$. Consequently $h_\beta$ is locally Lipschitz continuous on $X_\beta$.
\item[$(2)$] {}{Let $e(X_\beta)$ be the endpoint of $X_\beta$. Then at any point $t \in (0,\mathsf{d}(e(X_\beta),x))$ where $h_\beta$ is differentiable, the derivative of $\log h_\beta$ satisfies 
\begin{equation}\label{eqn4.6cccc}
-\frac{V_{K,N}'(\mathsf{d}(e(X_\beta),x)-t)}{V_{K,N}(\mathsf{d}(e(X_\beta),x)-t)}\leqslant (\log h_\beta)'(t)\leqslant \frac{V_{K,N}'(t)}{V_{K,N}(t)},
\end{equation}
(See \cite[Lem.~A.9]{CM21} for a proof.) }
\item[$(3)$]\label{feiohfwioeh} {}{By \cite[Cor. 4.19]{CM20}, $ \mathsf{d}_x \in D(\mathbf{\Delta}, X \setminus \{p\})$  and one element of $\mathbf{\Delta} \mathsf{d}_{x} \llcorner (X \setminus \{p\})$, that we still denote with $\mathbf{\Delta} \mathsf{d}_{x} \llcorner (X \setminus \{p\})$, is a Radon functional with the following representation formula:}

\begin{equation}\label{deltadistance}
{}{\mathbf{\Delta} \mathsf{d}_{x} \llcorner (X \setminus \{p\}) = (\log h_\beta)'\  \mathfrak{m} - \int_Q h_\beta \delta_{e(X_\beta)} \, \mathfrak{q}(d\beta).}
\end{equation}
{}{Here $\delta_{e(X_\beta)}$ denotes the Dirac measure at $e(X_\beta)$. Moreover, combining \cite[Prop.~5.1]{CM20} implies that for any open set $U\subset X$ such that $x\notin \bar{U}$ it holds $\|\mathbf{\Delta} \mathsf{d}_x\|(U)<\infty$, where $\|\mathbf{\Delta} \mathsf{d}_x\|$ is the total variation measure of $\mathbf{\Delta}\mathsf{d}_x$.}
\end{enumerate}
\end{remark}
\begin{remark}\label{GaussGreen}
{}{A direct consequence of Rmk.~\ref{rmk4.6aaaaaaaaa}(3) is, when $(X,\mathsf{d},\mathscr{H}^n)$ is a compact non-collapsed RCD$(K,n)$ space, for $\mathscr{L}^1$-a.e. $r\in (0,\mathrm{diam}(X,\mathsf{d}))$, the exterior unit normal to $\partial B_r(x)$ coincides $\mathscr{H}^{n-1}$-a.e. with $\nabla \mathsf{d}_x$ (see \cite[Prop. 6.1]{BPS23b}). Therefore combining the Gauss-Green formula \cite[Thm. 2.4]{BPS23} we see for any $g\in D(\Delta)$ it holds that 
\[
\int_{B_r(x)}\Delta g\,\mathrm{d}\mathscr{H}^n=\int_{\partial B_r(x)}\langle\nabla g,\nabla \mathsf{d}_x\rangle\,\mathrm{d}\mathscr{H}^{n-1},\ \mathscr{L}^1\text{-a.e.}\  r\in (0,\mathrm{diam}(X,\mathsf{d})).
\]
Moreover, due to \cite[Thm.~5.2]{BPS23b}, given any $y\neq x$ we have for $\mathscr{L}^1$-a.e. $r\in (0,\mathsf{d}(x,y))$ it holds
\[
\int_{B_r(x)}\mathbf{\Delta} \mathsf{d}_y=\int_{\partial B_r(x)}\langle\nabla \mathsf{d}_y,\nabla \mathsf{d}_x\rangle\,\mathrm{d}\mathscr{H}^{n-1},\ \mathscr{L}^1\text{-a.e.}\  r\in (0,\mathsf{d}(x,y)).
\]}
\end{remark}
\subsection{A H\"{o}lder estimate}
 
In this section, we prove a crucial H\"{o}lder estimate (Thm. \ref{thmhiojiofjea}) for the pointwise inner product of gradients of distance functions in the RCD framework. We would like to express our gratitude to Prof. Shouhei Honda for his insights regarding this result. We begin by recalling the Abresch-Gromoll inequality \cite{AG90}, which was generalized to RCD$(K,N)$ spaces in \cite{GM14} and \cite[Thm. 3.7, Cor. 3.8]{MN14}.

\begin{thm}[Abresch-Gromoll inequality]\label{AGTHM}There exists $\alpha(N) \in (0, 1)$ with the following property. Let $(X,\mathsf{d},\mathfrak{m})$ be an $\mathrm{RCD}(K,N)$ space. For any $p, q \in X$ with $l:= \mathsf{d}(p, q) \leq 1$, let $x_0$ the midpoint of a geodesic from $p$ to $q$. Then

\begin{equation}\label{AGthm} e_{p,q}(x) \leq C(K, N)\,l\, r^{1+\alpha(N)} , \quad \forall x \in B_{r}(x_0), \quad r \in (0, l/4), 
\end{equation}
where $e_{p,q}:= \mathsf{d}_p+\mathsf{d}_q-\mathsf{d}(p, q)$ is the so-called excess function associated to $p, q$.
\end{thm}
The following lemma constructs a test function using the heat kernel, whose properties are helpful in proving Thm. \ref{thmhiojiofjea}.
\begin{lem}\label{labelfejiaofjhaeoifj}
For any $L>100$, $N>2$, there exist a constant $C=C(L,N)$ such that for any $\mathrm{RCD}(-1,N)$ space $(X,\mathsf{d},\mathfrak{m})$ with two points $x,y\in X$ satisfying $\mathsf{d}(x,y)=L$ and $\mathfrak{m}(B_1(x))=1$, the function $P:z\mapsto \int_0^{1}\rho(y,z,t)\,\mathrm{d}t$ $($where $\rho$ is the heat kernel of the space$)$ is well-defined on $B_2(x)$ and satisfies 
\begin{equation}\label{fejaiofhjeaoifhjae}
C\geqslant \Delta P\geqslant C^{-1}\ \ \text{on}\ \ B_1(x);
\end{equation} 
\begin{equation}\label{feahiofhaoifhaeo}
C\Delta P\geqslant P\ \ \text{on}\ \ B_1(x).
\end{equation}

\end{lem}

\begin{proof}
By (\ref{JLZineq}) and Thm.~\ref{BGineq}, there exists a constant $J(N)$ such that for any $z\in B_2(x)$ it holds 
\[
\begin{aligned}
	&P(z)=\int_0^{1}\rho(y,z,t)\,\mathrm{d}t\leqslant J(N)\int_0^1 \frac{1}{\mathfrak{m}(B_{\sqrt{t}}(z))}\exp\left(-\frac{\mathsf{d}^2(z,y)}{5t}\right)\,\mathrm{d}t\\
	&=J(N)\,\frac{\mathfrak{m}(B_1(x))}{\mathfrak{m}(B_1(z))}\int_0^1 \frac{\mathfrak{m}(B_1(z))}{\mathfrak{m}(B_{\sqrt{t}}(z))}\exp\left(-\frac{\mathsf{d}^2(z,y)}{5t}\right)\,\mathrm{d}t\\
	&\leqslant J(N)\,\frac{\mathfrak{m}(B_3(z))}{\mathfrak{m}(B_1(z))}\int_0^1 \frac{\mathfrak{m}(B_1(z))}{\mathfrak{m}(B_{\sqrt{t}}(z))}\exp\left(-\frac{\mathsf{d}^2(z,y)}{5t}\right)\leqslant J(N)\int_0^1 t^{-\frac{N-1}{2}}\exp\left(-\frac{\mathsf{d}^2(z,y)}{5t}\right)\,\mathrm{d}t\\
	&\leqslant J(N)\int_0^1 t^{-\frac{N-1}{2}}\exp\left(-\frac{
	(L-1)^2}{5t}\right)\,\mathrm{d}t:=J(L,N).
\end{aligned}
\]

Therefore, $P$ is well-defined on $B_2(x)$. Observe that $\Delta P=\rho(\cdot,y,1)$ on $B_2(x)$. Similarly one can show that $J(L,N)\geqslant \Delta P\geqslant (J(L,N))^{-1}$ on $B_1(x)$. The conclusion then follows by letting $C(L,N)=(J(L,N))^2$.

\end{proof}
The proof of the following theorem employs the method described in \cite[Prop. 4.1]{H14}. See also \cite{CC962,MW19}.
\begin{thm}\label{thmhiojiofjea}
Let $(X,\mathsf{d},\mathscr{H}^n)$ be a non-collapsed $\mathrm{RCD}(K,n)$ space. Assume there exist points $x_i,y_i\in X$ such that $\mathsf{d}(x_i,y_i)=2l\leqslant 1$ $(i=1,2)$ and that a point $x_0\in X$ is the midpoint of a geodesic from $x_1$ to $y_1$ and a geodesic from $x_2$ to $y_2$. Then, for any 
\[
r\in \left(0,\min \left\{l/10000,l^{1/(1-\alpha(n))},|K|^{-1/2}/8\right\}\right) 
\] 
the following holds.
\begin{equation}\label{hjofiaehjoifahjeoifjh}
\fint_{B_r(x)}\left|\langle\nabla \mathsf{d}_{x_1},\nabla\mathsf{d}_{x_2}\rangle-\fint_{B_r(x)}\langle\nabla \mathsf{d}_{x_1},\nabla\mathsf{d}_{x_2}\rangle\,\mathrm{d}\mathscr{H}^n\right|\mathrm{d}\mathscr{H}^n\leqslant C(K,n)r^{\alpha(n)/2},\ \forall x\in B_r(x_0).
\end{equation}
Here $\alpha(n)$ is the constant taken in Thm.~\ref{AGTHM}.
\end{thm}
\begin{proof}Since RCD$(K,n)$ spaces also satisfy the RCD$(-|K|-1,n+2)$ condition, we may assume that $K<-1$ and $n>3$. Take $z\in X$ such that $\mathsf{d}(z,x)=8000r$ (such a point exists, for instance on the geodesic from $x_i$ to $x$). For any $0<r\leqslant |K|^{-1/2}/8$,  since the rescaled space $(X,(8r)^{-1}\mathsf{d},(\mathscr{H}^n(B_{8r}(x)))^{-1}\mathscr{H}^n)$ is an RCD$(64r^2K,n)$ space, Lem. \ref{labelfejiaofjhaeoifj} guarantees the existence of $G\in D(\Delta,B_{8r}(x))$ satisfying
\begin{equation}\label{hfioaehfoahweouifeah}
0<G\leqslant C(n)\ \  \text{and}\ \  C(n)^2r^{-2}\geqslant C(n)\Delta G\geqslant r^{-2}\ \text{on}\ B_{8r}(x).
\end{equation}

Let $ u_i$, $ v_i$ be solutions to the following auxiliary problems, where by $H_0^{1,2}(B_{4r}(x))$ we mean the $H^{1,2}$-closure of $\mathrm{Lip}_c(B_{4r}(x),\mathsf{d})$. The existence of such solutions is guaranteed by \cite[Cor. 1.2]{BM95} and \cite[Thm. 10.12]{BB11}.
\[
\left\{
\begin{aligned}
\Delta  u_i=0\ \  \text{on}\  B_{4r}(x),\ \ \ \ \ \ \ \ \ \\
\mathsf{d}_{x_i}-l- u_i\in H^{1,2}_0(B_{4r}(x)),
\end{aligned}
\right.
\,\, \left\{
\begin{aligned}
\Delta  v_i=0\ \  \text{on}\  B_{4r}(x),\ \ \ \ \ \ \ \ \ \\
 \mathsf{d}_{y_i}-l-v_i\in H^{1,2}_0(B_{4r}(x)).
\end{aligned}
\right.
\]

We assert: 
\begin{equation}\label{feawhjuiofhouawehf}
\|\mathsf{d}_{x_i}-l- u_i\|_{L^\infty(B_{4r}(x))}\leqslant C(K,n)\, r^{1+\alpha}.
\end{equation} 

From (\ref{AGthm}), we derive
\begin{equation}\label{ttcomestogz1}
0\leqslant\min_{\bar{B}_{4r}(x)} (\mathsf{d}_{x_i}+\mathsf{d}_{y_i}-2l)\leqslant  \max_{\bar{B}_{4r}(x)} (\mathsf{d}_{x_i}+\mathsf{d}_{y_i}-2l)=\max_{\bar{B}_{4r}(x)}e_{x_i,y_i}\leqslant C(K,n)\, r^{1+\alpha}.
\end{equation}
Applying the weak maximum and minimum principles (~\cite[Lem.~8.32]{BB11}, \cite[Thm.~2.5]{GR19}) yields
\begin{equation}\label{ttcomestogz2}
 0\leqslant \sup_{B_{4r}(x)} ( u_i+ v_i)\leqslant  C(K,n)\, r^{1+\alpha}.
\end{equation}

Due to (\ref{deltadistance}), the measure valued Laplacians satisfy
\begin{align}\label{fneoaihfaef}
\mathbf{\Delta} \mathsf{d}_{x_i}\llcorner_{B_{l/4}(x_0)}\leqslant  \frac{C(K,n)}{l}\mathscr{H}^n\llcorner_{B_{l/4}(x_0)},\,\,\mathbf{\Delta} \mathsf{d}_{y_i}\llcorner_{B_{l/4}(x_0)}\leqslant  \frac{C(K,n)}{l}\mathscr{H}^n\llcorner_{B_{l/4}(x_0)}.
\end{align}
Using (\ref{hfioaehfoahweouifeah}) then shows \[
 \mathbf{\Delta}(\mathsf{d}_{x_i}- u_i-C(K,n)\,l^{-1}r^2G)\llcorner_{B_{4r}(x)}\leqslant 0,\  \mathbf{\Delta}(\mathsf{d}_{y_i}- v_i-C(K,n)\,l^{-1}r^2G)\llcorner_{B_{4r}(x)}\leqslant 0.
\]By the weak minimum principle, 
\[
\min\left\{\inf_{B_{4r}(x)}(\mathsf{d}_{x_i}-l- u_i),\inf_{B_{4r}(x)}(\mathsf{d}_{y_i}-l- v_i)\right\}\geqslant -C(K,n)\,l^{-1}r^2\sup_{\partial B_{4r}(x)}G.
\]
Combining this with (\ref{hfioaehfoahweouifeah}),  (\ref{ttcomestogz1}) and (\ref{ttcomestogz2}) implies
\[
-C(K,n)\,l^{-1}r^2\leqslant \mathsf{d}_{x_i}-l- u_i\leqslant C(K,n)\,(r^{1+\alpha}+l^{-1}r^2).
\]Since $l^{-1}r\leqslant r^{\alpha}$, (\ref{feawhjuiofhouawehf}) follows.

 Because $s\mapsto \mathscr{H}^n(B_s(x))/V_{K,n}(s)$ is monotone non-increasing, by taking derivative and using co-area formula we see
\begin{align}\label{b4.15}
\frac{\mathscr{H}^{n-1}(\partial B_s(x))}{\mathscr{H}^{n}(B_s(x))}\leqslant\frac{ V_{K,n}'(s)}{V_{K,n}(s)}\leqslant \frac{C(K,n)}{s},\ \mathscr{L}^1\text{-a.e.}\ s\in (0,r).
\end{align}
Rmk.~\ref{GaussGreen} then implies for $\mathscr{L}^1$-a.e. $s\in (0,r)$ it holds 
\begin{align}\label{feaufhaefhaeio}
\left|\int_{B_s(x)}\mathbf{\Delta}\mathsf{d}_{x_i}\right|\leqslant \mathscr{H}^{n-1}(\partial B_s(x))\leqslant C(K,n)\,\frac{\mathscr{H}^n(B_s(x))}{s},\ \mathscr{L}^1\text{-a.e.}\ s\in (0,r).
\end{align}

Using the harmonicity of $u_i$, (\ref{feawhjuiofhouawehf}), (\ref{fneoaihfaef}) and (\ref{feaufhaefhaeio}), we compute
\begin{equation}\label{faehuifhaoeuif}
\begin{aligned}
&\fint_{B_{4r}(x)}{|\nabla (\mathsf{d}_{x_i}- u_i)|}^2\,\mathrm{d}\mathscr{H}^n\\
=\ &\fint_{B_{4r}(x)}\left\langle\nabla \mathsf{d}_{x_i},\nabla(\mathsf{d}_{x_i}-l- u_i)\right\rangle\,\mathrm{d}\mathscr{H}^n=\fint_{B_{4r}(x)}(\mathsf{d}_{x_i}-l- u_i)\mathbf{\Delta}\,\mathsf{d}_{x_i}\\
=\ &\fint_{B_{4r}(x)}\left(\mathsf{d}_{x_i}-l- u_i+2C(K,n)r^{1+\alpha}\right)\mathbf{\Delta}\,\mathsf{d}_{x_i}-2C(K,n)r^{1+\alpha}\fint_{B_{4r}(x)}\mathbf{\Delta}\,\mathsf{d}_{x_i}\leqslant C(K,n)\,r^{\alpha}.
\end{aligned}
\end{equation}

Let $\varphi$ be a cut-off function supported in $B_{4r}(x)$ with $\varphi=1$ on $B_{2r}(x)$ (see Rmk. \ref{rmk2.16}). Since $\|\Delta \varphi\|_{L^\infty}\leqslant C(K,n)r^{-2}$ and $u_i$ is harmonic, substituting (\ref{faehuifhaoeuif}) into  (\ref{abc2.14}) and applying H\"{o}lder's inequality give

\[
\begin{aligned}
&\int_{B_{2r}(x)}{|\text{Hess}\, u_i|_{\mathsf{HS }}}^{2}\,\mathrm{d}\mathscr{H}^n\leqslant \int_{B_{4r}(x)}\varphi{|\text{Hess}\, u_i|_{\mathsf{HS}}}^{2}\,\mathrm{d}\mathscr{H}^n\leqslant \int_{B_{4r}(x)}(\Delta\varphi/2+|K|)\,{|\nabla u_i|}^{2}\,\mathrm{d}\mathscr{H}^n\\
&=\int_{B_{4r}(x)}(\Delta\varphi/2+|K|)\left({|\nabla(\mathsf{d}_{x_i}-u_i)|}^{2}+1-2\langle\nabla(\mathsf{d}_{x_i}- u_i),\nabla\mathsf{d}_{x_i}\rangle\right)\,\mathrm{d}\mathscr{H}^n\\
&=\int_{B_{4r}(x)}(\Delta\varphi/2+|K|)\left({|\nabla(\mathsf{d}_{x_i}-u_i)|}^{2}-2\langle\nabla(\mathsf{d}_{x_i}- u_i),\nabla\mathsf{d}_{x_i}\rangle\right)\,\mathrm{d}\mathscr{H}^n+|K|\,\mathscr{H}^n(B_{4r}(x))\\
&\leqslant C(K,n)\,\mathscr{H}^n(B_{4r}(x))\, r^{\alpha/2-2}\leqslant C(K,n)\,\mathscr{H}^n(B_{2r}(x))\, r^{\alpha/2-2},
\end{aligned}
\]
where in the last inequality we use Thm. \ref{BGineq}.

 Recalling $\sup_{B_{4r}(x)}|\mathsf{d}_{x_i}-l|\leqslant 5r$, we see $\sup_{B_{4r}(x)}| u_i|\leqslant C(K,n)\,r$. Applying \cite[Thm.~1.1]{J14} then shows $\|\nabla u_i\|_{L^\infty(B_{2r}(x))}\leqslant C(K,n)$. Thus we obtain
\begin{equation}\label{fahuoefhoaewhfio}
\begin{aligned}
&\fint_{B_r(x)}\left|\langle\nabla  u_1,\nabla u_2\rangle-\fint_{B_r(x)}\langle\nabla  u_1,\nabla u_2\rangle\,\mathrm{d}\mathscr{H}^n\right|\mathrm{d}\mathscr{H}^n\\
\leqslant\ &C(K,n)\, r \left(\fint_{B_{2r}(x)}|\nabla \langle\nabla  u_1,\nabla  u_2\rangle|\mathrm{d}\mathscr{H}^n\right)^{1/2}\ (\text{(1,2)-Poincar\'{e} inequality})\\
\leqslant \ &C(K,n)\, r \left(\fint_{B_{2r}(x)}\left(|\mathrm{Hess}\, u_1|_{\mathsf{HS}}|\nabla u_2|+|\mathrm{Hess}\, u_2|_{\mathsf{HS}}|\nabla u_1|\right)\mathrm{d}\mathscr{H}^n\right)^{1/2}\ (\text{by (\ref{11eqn2.16})})\\
\leqslant \ &C(K,n)\, r \left(\fint_{B_{2r}(x)}\left(|\mathrm{Hess}\, u_1|_{\mathsf{HS}}+|\mathrm{Hess}\, u_2|_{\mathsf{HS}}\right)\mathrm{d}\mathscr{H}^n\right)^{1/2}\leqslant C(K,n)\, r^{\alpha/4}
\end{aligned}
\end{equation}

Finally, from (\ref{faehuifhaoeuif}) and Thm. \ref{BGineq} we know the difference between (\ref{hjofiaehjoifahjeoifjh}) and (\ref{fahuoefhoaewhfio}) is controlled by $C(K,n)\, r^{\alpha/4}$, completing the proof.
\end{proof} 
\begin{remark}
The theory of sets of finite perimeter allows for a proof of this theorem in general $\mathrm{RCD}(K,N)$ spaces. In fact, a generalized version of (\ref{b4.15}) and (\ref{feaufhaefhaeio}) suffices to establish the result.
\end{remark}
\subsection{Proof of the main theorem: smoothness under volume homogeneity}

For reader's convenience, we recall Thm. \ref{thm1.11} as follows.
\begin{thm}\label{thm5.1}
Let $(X,\mathsf{d},\mathfrak{m})$ be an $\mathrm{RCD}(K,N)$ space with $n=\mathrm{dim}_{\mathsf{d},\mathfrak{m}}(X)$. If the volume of the intersection of any two geodesic balls depends only on their radii and the distance between their centers, then:
\begin{enumerate}
\item[$(1)$] $\mathfrak{m}=c\mathscr{H}^n$ for some $c>0$ and $(X,\mathsf{d},\mathscr{H}^n)$ is a non-collapsed $\mathrm{RCD}(K,n)$ space.
\item[$(2)$] If $(X,\mathsf{d})$ is compact, then it is isometric to an $n$-dimensional smooth closed Riemannian manifold $(M^n,\mathrm{g})$.
\item[$(3)$] If $(X,\mathsf{d})$ is compact and simply-connected, then it is a harmonic manifold in the classical sense. 
\end{enumerate}

\end{thm}

\begin{remark}\label{rmk4.8aaa}
By approximating with characteristic functions, the assumption in Thm. \ref{thm5.1} can be equivalently stated as follows. For any two compactly supported $L^2$-integrable functions $h_1,h_2$ on $\mathbb{R}$, the convolution of corresponding radial kernel functions $H_i=h_i\circ \mathsf{d}$ ($i=1,2$) remains radial. This generalizes Thm.  \ref{harm3} to the RCD framework.

\end{remark}
\subsubsection{Measure structure}
\begin{proof}[Proof of Thm. \ref{thm5.1} (1)]
By the volume homogeneity assumption, there exists a function $\theta:(0,\mathrm{diam}(X,\mathsf{d}))\rightarrow [0,\infty)$ such that
\begin{align}\label{fhoeahofihaefi}
\mathfrak{m}(B_r(x))=\theta(r), \ \forall x\in X,\ \forall r\in (0,\mathrm{diam}(X,\mathsf{d})).
\end{align}
Thm. \ref{BS} and Rmk. \ref{AHT} imply that the limit (\ref{vpxbbctt}) exists with a constant $c>0$.
\begin{equation}\label{vpxbbctt}
\lim_{r\rightarrow 0}\frac{\mathfrak{m}(B_r(x))}{\omega_n r^n}=\lim_{r\rightarrow 0}\frac{\theta(r)}{\omega_n r^n}=c.
\end{equation}

Since (\ref{eeeeeeeeeeqn3.14}) also holds, the conclusion follows by applying \cite[Thm.~2.4.3]{AT04}, \cite[Thms.~1.5 and 2.22]{BGHZ23} as in the proof of Thm. \ref{mainthm2} (2).
\end{proof}
\begin{remark}
A metric measure space $(X,\mathsf{d},\mathfrak{m})$ is called ball-homogeneous if there exists a function $\theta: [0,\infty)\to [0,\infty)$ such that
\[
\mathfrak{m}(B_r(x))=\theta(r), \, \forall x\in X,\ \forall r\in [0,\infty).
\]
To the best of the our knowledge, the following proposition provides the state-of-the-art in the regularity of ball-homogeneous RCD$(K,N)$ spaces. Nevertheless, it remains an open question whether a compact ball-homogeneous RCD$(K,N)$ space is isometric to a smooth closed Riemannian manifold, or even a Lipschitz manifold.
\end{remark}
\begin{prop}
Assume $(X,\mathsf{d},\mathfrak{m})$ is a ball-homogeneous $\mathrm{RCD}(K,N)$ space with $\mathrm{dim}_{\mathsf{d},\mathfrak{m}}(X)=n$. Then 
\begin{enumerate}
\item[$(1)$] The conclusion of the first statement of Thm. \ref{thm5.1} also holds, i.e., $\mathfrak{m}=c\mathscr{H}^n$ for some $c>0$ and $(X,\mathsf{d},\mathscr{H}^n)$ is a non-collapsed $\mathrm{RCD}(K,n)$ space.
\item[$(2)$] $X$ is a $C^\alpha$-manifold, where $\alpha\in (0,1)$ is a H\"{o}lder exponent determined by $n$.
\end{enumerate} 
\end{prop}
\begin{proof}
{}{For the second assertion, we combine (\ref{vpxbbctt}) with \cite[Thm.~1.6]{DG18}, and the Reifenberg's Thm. \cite[Thm. A.1.2]{ChCo1} (see also \cite{CJN21})}.
\end{proof}

 
\subsubsection{Smoothness}
 In this section, we focus on the second statement of Thm.~\ref{thm5.1} (Thm.~\ref{thm1.11}). Let $(X,\mathsf{d},\mathscr{H}^n)$ be a compact non-collapsed $\mathrm{RCD}(K,n)$ space with $n\geqslant 2$, where the volume of the intersection of any two geodesic balls depends only on their radii and the distance between their centers. Let $D=\mathrm{diam}(X,\mathsf{d})$. For any two points $x,y\in X$, we denote by $\gamma_{x,y}$ a unit speed minimal geodesic from $x$ to $y$.

\begin{prop}[Local extendibility of geodesics]\label{step1} 
Any geodesic in $X$ with length less than $D$ is locally extendible.
\end{prop}
\begin{proof}
Fix $x_0,y_0\in X$ with $\mathsf{d}\left(x_0,y_0\right)=D$ and let $\gamma_0:=\gamma_{x_0,y_0}$. For $r,s,t>0$ satisfying $r+s+t<D$, Thm. \ref{BGineq} yields $\mathscr{H}^n(\partial B_{r+s}(x_0))=0$. Thus
{}{\begin{equation}\label{ttt4.5}
\begin{aligned}
&\mathscr{H}^n(B_{r+t}(\gamma_0(s))\setminus \bar{B}_{r+s}(x_0))=\mathscr{H}^n(B_{r+t}(\gamma_0(s))\setminus B_{r+s}(x_0))\\
=&\mathscr{H}^n(B_{r+t}(\gamma_0(s)))-\mathscr{H}^n(B_{r+t}(\gamma_0(s))\cap B_{r+s}(x_0)).
\end{aligned}
\end{equation}}Since $\gamma_0(r+s+t/2)\in B_{r+t}(\gamma_0(s))\setminus \bar{B}_{r+s}(x_0)$, which is an open set, {}{there exists $\delta=\delta(r,s,t)$ such that $B_\delta(\gamma_0(r+s+t/2))\subset B_{r+t}(\gamma_0(s))\setminus \bar{B}_{r+s}(x_0)$. Therefore \[\mathscr{H}^n\left( B_{r+t}(\gamma_0(s))\setminus \bar{B}_{r+s}(x_0)\right)>\mathscr{H}^n(B_\delta(\gamma_0(r+s+t/2))>0.\]}

Let $x,y\in X$ with $l:=\mathsf{d}\left(x,y\right)<D$ and $\gamma:=\gamma_{x,y}$. For $l' \in \left(0,\min \{l,D-l\} /4\right)$ and large $i\in\mathbb{N}$, (\ref{ttt4.5}) and the volume homogeneity assumption imply  
\[
\mathscr{H}^n(B_{l+1/i}(\gamma(l'))\setminus B_{l+l'}(x))>0.
\]
Thus $B_{l+1/i}(\gamma(l'))\setminus B_{l+l'}(x)$ is non-empty. Select $z_i\in  B_{l+1/i}(\gamma(l'))\setminus B_{l+l'}(x)$. By compactness, after passing to a subsequence, $\{z_i\}$ converges to some $z\in X$ with $\mathsf{d}(z,x)\geqslant l+l'$. We deduce $\mathsf{d}(z,x)=l+l'$ by noticing that
\[
l+l'\geqslant \lim_{i\to\infty}\left(l+\frac{1}{i}+l'\right)\geqslant \lim_{i\to\infty}\mathsf{d}(z_i,\gamma(l'))+l'\geqslant \lim_{i\to\infty}\mathsf{d}(z_i,x)=\mathsf{d}(z,x).
\]

Let $\gamma':=\gamma_{y,z}: [0,l']\rightarrow X$, then for any $s\in [0,l]$, $s'\in [0,l']$ it holds that
\[
\begin{aligned}
l+l'&= \mathsf{d}\left(x,z\right)\leqslant \mathsf{d}\left(x,\gamma(s)\right)+\mathsf{d}\left(\gamma(s),\gamma'({}{s'})\right)+\mathsf{d}\left(\gamma'({}{s'}),z\right)\\
   \ &=\mathsf{d}\left(\gamma(s),\gamma'({}{s'})\right)+s+l'-s'.
\end{aligned}
\]
Conversely, $\mathsf{d}\left(\gamma(s),\gamma'({}{s'})\right)\leqslant \mathsf{d}\left(\gamma(s),y\right)+\mathsf{d}\left(y,\gamma'({}{s'})\right)=l-s+s'$. Hence, $\mathsf{d}\left(\gamma(s),\gamma'({}{s'})\right)=l-s+s'$, implying  the following curve is a minimal geodesic from $x$ to $z$.
\[
\tilde{\gamma}(s)=
\left\{\begin{array}{ll}
\gamma(s) &\text{if $s\in [0,l]$},\\
\gamma'(s-l')&\text{if $s\in (l,l+l']$.}
\end{array}\right.
\]

\end{proof}
{}{\begin{remark}[Uniqueness of geodesics]\label{step2}
The geodesic from $x$ to $y$ is unique. Suppose there exists another unit-speed geodesic $\tau\neq \gamma$, then the curve
\[
\tilde{\tau}(s)=
\left\{\begin{array}{ll}
\tau(s) &\text{if $s\in [0,l]$},\\
\gamma'(s-l')&\text{if $s\in (l,l+l']$,}
\end{array}\right.
\]
would also be a minimal geodesic from $x$ to $z$, contradicting Thm. \ref{RCDnb}. Thus the limit point $z$ is unique and independent of the choice of sequence $\{z_i\}$. Moreover, any geodesic with length $l<D$ can be uniquely extended to length $D$. 
\end{remark}}



\begin{cor}\label{step3} For any $x_0\in X$, the disintegration formula $(\ref{disintefor})$ holds for $\mathscr{H}^n\llcorner_{\mathcal{T}_{x_0}}$ with respect to $\mathsf{d}_{x_0}$, where  $\mathcal{T}_{x_0}=X\setminus (\{x_0\}\cup\partial B_D(x_0))$ and $Q_{x_0}=\partial B_{D/3}(x_0)$.
\end{cor}
\begin{proof}

For $0<r_1<r_2<D$, the map
\[
\begin{aligned}
F_{r_2\rightarrow r_1}:\partial B_{r_2}(x_0)&\longrightarrow \partial B_{r_1}(x_0)\\
y&\longmapsto \gamma_{x_0,y}(r_1),
\end{aligned}
\]
is a well-defined bijection due to Prop. \ref{step1} and Rmk. \ref{step2}. If $\partial B_{r_2}(x_0)\ni y_i\rightarrow y\in \partial B_{r_2}(x_0)$, Rmk. \ref{step2} implies that $\gamma_{x_0,y_i}$ converges to $\gamma_{x_0,y}$. Consequently, $F_{r_2\rightarrow r_1}(y_i)\rightarrow F_{r_2\rightarrow r_1}(y)$, yielding the continuity of $F_{r_2\rightarrow r_1}$. The inverse map $F_{r_1\rightarrow r_2}:=\left( F_{r_2\rightarrow r_1}\right)^{-1}$ is also continuous, making $F_{r_2\rightarrow r_1}$ a homeomorphism. Indeed, due {}{to the compactness of the spheres}, both $F_{r_2\rightarrow r_1}$ and $F_{r_1\rightarrow r_2}$ are uniform continuous.

{}{By the definition of the equivalence relation $\mathcal{R}_{x_0}$ over $\mathcal{T}_{x_0}$, and Rmk. \ref{step2}, each $x\in \mathcal{T}_{x_0}$ uniquely determines a point $\beta=F_{\mathsf{d}\left(y,x_0\right)\rightarrow D/3}(x)\in \partial B_{D/3}(x_0)$ such that $(x,\beta)\in \mathcal{R}_{x_0}$. Furthermore, for any $r\in (0,D)$, $(\beta, F_{D/3\to r}(\beta))$ also lies in $\mathcal{R}_{x_0}$. Thus we have the partition $\mathcal{T}_{x_0}=\bigcup_{\beta\in \partial B_{D/3}(x_0)}X_\beta$, where each $X_\beta$ is a geodesic:
\[
\begin{aligned}
X_\beta: (0,D)&\longrightarrow X\\
t&\longmapsto F_{D/3\to t}(\beta). 
\end{aligned}
\]}The quotient map
\[
\begin{aligned}
\mathfrak{Q}_{x_0}:\mathcal{T}_{x_0}&\longrightarrow Q_{x_0}:=\partial B_{D/3}(x_0)\\
x&\longmapsto F_{\mathsf{d}\left(x,x_0\right)\rightarrow D/3}(x),
\end{aligned}
\]
is continuous and therefore Borel measurable. Using the identification $\mathcal{T}_{x_0}\ni y\mapsto (F_{\mathsf{d}(x_0,y)\to D/3}(y),\mathsf{d}(x_0,y))$, recalling $\mathscr{H}^n(X\setminus \mathcal{T}_{x_0})=0$, and letting $\sigma_{x_0}(\beta,\cdot)$ denote the CD$(K,n)$ density on each $X_{\beta}$, we are able to rewrite the disintegration formula as
\begin{equation}\label{eqn4.10aabbbb}
\int_{X} f\mathop{\mathrm{d}\mathscr{H}^n}=\int_{Q_{x_0}}\int_{(0,D)}f\left(\beta,t\right)\sigma_{x_0}\left(\beta,t\right)\mathop{\mathrm{d}t}\mathop{\mathrm{d}\mathfrak{q}_{x_0}(\beta),}\ \forall f\in L^1.
\end{equation}
\end{proof}
\begin{remark}
For any $x,y\in X$ with $0<\mathsf{d}(x,y)<D$, by (\ref{deltadistance}) we know when $r=\min\{\mathsf{d}(x,y),D-\mathsf{d}(x,y)\}/4$ it holds
\begin{equation}\label{fdahfuafhaeiohfo}
\mathsf{d}_x\in D(\Delta,B_r(y)) \,\,\text{with} \, \|\Delta \mathsf{d}_x\|_{L^\infty(B_r(y))}\leqslant C(K,n,D)\max\left\{\frac{1}{\mathsf{d}(x,y)},\frac{1}{D-\mathsf{d}(x,y)}\right\}.
\end{equation}

\end{remark}
\begin{remark}\label{rmk4.16}
{}{By co-area formula, for any Borel set $A\subset \partial B_{D/3}(x_0)$ and $\mathscr{L}^1$-a.e. $R\in (0,D)$ it holds
\begin{equation}\label{eqn4.9fdafijoaje}
\int_{\partial B_{R}(x_0)}F_{D/3\to R}(A)\mathop{\mathrm{d}\mathscr{H}^{n-1}}=\int_{A}\sigma_{x_0}(\cdot,R)\mathop{\mathrm{d}\mathfrak{q}_{x_0}}.
\end{equation}
For any two radii $0<R_1<R_2<D$, Thms. \ref{BGineq} and \ref{DG18cor1.7}, Fubini's Thm. and (\ref{eqn4.6cccc}),  (\ref{eqn4.10aabbbb}) yield
\[
\begin{aligned}
&\left|\int_{A}\sigma_{x_0}(\cdot,R_2)\mathop{\mathrm{d}\mathfrak{q}_{x_0}}-\int_{A}\sigma_{x_0}(\cdot,R_1)\mathop{\mathrm{d}\mathfrak{q}_{x_0}}\right|=\left|\int_{A}\int_{R_1}^{R_2}\frac{\partial \sigma_{x_0}}{\partial t}\,\mathrm{d}t\mathrm{d}\mathfrak{q}_{x_0}\right|\\
\leqslant\ & \int_{R_1}^{R_2}\int_{A}\left|\frac{\partial \log\sigma_{x_0}}{\partial t}\right|
\sigma_{x_0}\,\mathrm{d}t\mathrm{d}\mathfrak{q}_{x_0}\leqslant \int_{R_1}^{R_2}\int_{Q_{x_0}}\left|\frac{\partial \log\sigma_{x_0}}{\partial t}\right|
\sigma_{x_0}\,\mathrm{d}t\mathrm{d}\mathfrak{q}_{x_0}\\
\leqslant \ &C(K,n)\max\{{R_1}^{-1},(D-R_2)^{-1}\}\mathscr{H}^n(A_{R_1,R_2}(x_0))\\\leqslant\ & C(K,n)\max\{{R_1}^{-1},(D-R_2)^{-1}\}\frac{\mathscr{H}^n(B_{R_1}(x_0))}{V_{K,n}(R_1)}(V_{K,n}(R_2)-V_{K,n}(R_1))\\
\leqslant\ & C(K,n,D)\max\{{R_1}^{-1},(D-R_2)^{-1}\}(R_2-R_1).\\
\end{aligned}
\]
Thus $R\mapsto \int_{A}\sigma_{x_0}(\cdot,R)\mathop{\mathrm{d}\mathfrak{q}_{x_0}}$ is locally Lipschitz continuous on $(0,D)$. Combining (\ref{eqn4.9fdafijoaje}) with co-area formula then implies that $R\mapsto \mathscr{H}^{n-1}(F_{D/3\to R}(A))$ is also locally Lipschitz continuous. Hence (\ref{eqn4.9fdafijoaje}) holds for every $R\in (0,D)$.}

\end{remark}
 


\begin{prop}\label{fehohfaoehf}
Fix $x_0\in X$, $D_0\in (0,D)$. For any $p,q\in \partial B_{D_0}(x_0)$, the map $x\mapsto \langle\nabla\mathsf{d}_p,\nabla\mathsf{d}_q\rangle(x)$ is $\alpha$-H\"{o}lder continuous on $B_{r_0}(x_0)$ for some $\alpha=\alpha(n)$, $r_0=r_0(K,n,D_0)$.
\end{prop}
\begin{proof}
Let $r_1=\min\{D_0/1000,(D-D_0)/1000,1\}$. According to Rmk. \ref{step2}, for any $y\in B_{r_1}(x_0)$, there exist $p',q'\in X$ such that $y$ is both the midpoint of $\gamma_{p,p'}$ and $\gamma_{q,q'}$. Thm. \ref{thmhiojiofjea} guarantees the existence of $r_0=r_0(K,n,D_0)<r_1$ and $\tilde{\alpha}(n)<1$ such that for all $r\in (0,r_0)$ and $x\in B_{r}(y)$ it holds
\begin{align}\label{fehjofiheoiaf}
\fint_{B_r(x)}\left|\langle\nabla \mathsf{d}_{p},\nabla\mathsf{d}_{q}\rangle-\fint_{B_r(x)}\langle\nabla \mathsf{d}_{p},\nabla\mathsf{d}_{q}\rangle\,\mathrm{d}\mathscr{H}^n\right|\mathrm{d}\mathscr{H}^n\leqslant C(K,n)\, r^{\tilde{\alpha}(n)}.
\end{align}
Indeed, (\ref{fehjofiheoiaf}) holds for all $r\in (0,r_0)$ and $x\in B_{r_0}(x_0)$ since letting $y=\gamma_{x,x_0}(r/2)$ implies $x\in B_r(y)$ and $y\in B_{r_0}(x_0)$. Now for any $0<s<t<r_0$, (\ref{fehjofiheoiaf}) as well as Thm. \ref{BGineq} implies: 
\begin{equation}\label{fueahofahoe}
\begin{aligned}
&\left|\fint_{B_{s}(x)}\langle\nabla \mathsf{d}_{p},\nabla\mathsf{d}_{q}\rangle\,\mathrm{d}\mathscr{H}^n-\fint_{B_{t}(x)}\langle\nabla \mathsf{d}_{p},\nabla\mathsf{d}_{q}\rangle\,\mathrm{d}\mathscr{H}^n\right|\\
\leqslant\ &\fint_{B_{s}(x)}\left|\langle\nabla \mathsf{d}_{p},\nabla\mathsf{d}_{q}\rangle-\fint_{B_{t}(x)}\langle\nabla \mathsf{d}_{p},\nabla\mathsf{d}_{q}\rangle\,\mathrm{d}\mathscr{H}^n\right|\mathrm{d}\mathscr{H}^n\\
\leqslant\ &\frac{\mathscr{H}^n(B_{t}(x))}{\mathscr{H}^n(B_{s}(x))}\fint_{B_{t}(x)}\left|\langle\nabla \mathsf{d}_{p},\nabla\mathsf{d}_{q}\rangle-\fint_{B_{t}(x)}\langle\nabla \mathsf{d}_{p},\nabla\mathsf{d}_{q}\rangle\,\mathrm{d}\mathscr{H}^n\right|\mathrm{d}\mathscr{H}^n\leqslant C(K,n)\, {t}^{n+\tilde{\alpha}}{s}^{-n}.
\end{aligned}
\end{equation}
This ensures the uniform convergence of $\lim_{i\to\infty}\fint_{B_{2^{-i}}(x)}\langle\nabla \mathsf{d}_{p},\nabla\mathsf{d}_{q}\rangle\,\mathrm{d}\mathscr{H}^n$ on $B_{r_0}(x_0)$. Given any sequence $s_i\to 0$, since there exists $\{a_i\}\subset \mathbb{N}$ such that $2^{a_i}\leqslant {s_i}^{-1}\leqslant 2^{a_i+1}$, it then follows from (\ref{fueahofahoe}) that 
\[
\lim_{i\to\infty}\fint_{B_{2^{-i}}(x)}\langle\nabla \mathsf{d}_{p},\nabla\mathsf{d}_{q}\rangle\,\mathrm{d}\mathscr{H}^n=\lim_{i\to\infty}\fint_{B_{s_i}(x)}\langle\nabla \mathsf{d}_{p},\nabla\mathsf{d}_{q}\rangle\,\mathrm{d}\mathscr{H}^n:=F(x).
\]

Applying (\ref{fueahofahoe}) again, for any $r<r_0$ and $x_1,x_2\in B_{r_0}(x_0)$ we obtain:
\[
\begin{aligned}
&\left|\fint_{B_{r}(x_i)}(\langle\nabla \mathsf{d}_{p},\nabla\mathsf{d}_{q}\rangle-F(x_i))\,\mathrm{d}\mathscr{H}^n\right|\\
\leqslant\ &\sum_{j=0}^\infty\left|\fint_{B_{2^{-j} r}(x_i)}\langle\nabla \mathsf{d}_{p},\nabla\mathsf{d}_{q}\rangle\,\mathrm{d}\mathscr{H}^n-\fint_{B_{2^{-(j+1)}r}(x_i)}\langle\nabla \mathsf{d}_{p},\nabla\mathsf{d}_{q}\rangle\,\mathrm{d}\mathscr{H}^n\right|\leqslant\, C(K,n)\, r^{\tilde{\alpha}}.
\end{aligned}
\]
Assume $\mathsf{d}(x_1,x_2)<r$. Combining the above inequality with Thm. \ref{BGineq} and the volume homogeneity assumption, we derive:
\[
\begin{aligned}
&|F(x_1)-F(x_2)|\leqslant \left|\fint_{B_r(x_1)}\langle\nabla \mathsf{d}_{p},\nabla\mathsf{d}_{q}\rangle\,\mathrm{d}\mathscr{H}^n-\fint_{B_r(x_2)}\langle\nabla \mathsf{d}_{p},\nabla\mathsf{d}_{q}\rangle\,\mathrm{d}\mathscr{H}^n\right|\\
&+\sum_{i=1,2}\left|\fint_{B_{r}(x_i)}(\langle\nabla \mathsf{d}_{p},\nabla\mathsf{d}_{q}\rangle-F(x_i))\,\mathrm{d}\mathscr{H}^n\right|\\
\leqslant\  &\frac{1}{\mathscr{H}^n(B_r(x_2))}\int_{B_r(x_1)\triangle B_r(x_2)}|\langle\nabla \mathsf{d}_{p},\nabla\mathsf{d}_{q}\rangle|\,\mathrm{d}\mathscr{H}^n+\sum_{i=1,2}\left|\fint_{B_{r}(x_i)}(\langle\nabla \mathsf{d}_{p},\nabla\mathsf{d}_{q}\rangle-F(x_i))\,\mathrm{d}\mathscr{H}^n\right|\\
\leqslant\ &\frac{\mathscr{H}^n(B_r(x_1)\triangle B_r(x_2))}{\mathscr{H}^n(B_r(x_2))}+C(K,n)\,r^{\tilde{\alpha}}\leqslant\ \frac{\mathscr{H}^n(A_{r-\mathsf{d}(x_1,x_2),r+\mathsf{d}(x_1,x_2)}(x_2))}{\mathscr{H}^n(B_r(x_2))}+C(K,n)\,r^{\tilde{\alpha}}\\
\leqslant\ & C(K,n)\left(\left(1+\frac{\mathsf{d}(x_1,x_2)}{r}\right)^n-\left(1-\frac{\mathsf{d}(x_1,x_2)}{r}\right)^n+r^{\tilde{\alpha}}\right).
\end{aligned}
\]

Finally, for any $x_1,x_2\in B_{r_0}(x_0)$ with $\mathsf{d}(x_1,x_2)\leqslant {r_0}^2<1$, plugging $r=\sqrt{\mathsf{d}(x_1,x_2)}$ into the above inequality yields 
\[
|F(x_1)-F(x_2)|\leqslant C(K,n) {(\mathsf{d}(x_1,x_2))}^\alpha,\ \ \text{where}\ \alpha=\tilde{\alpha}/2.
\] 
By Lebesgue differentiation Thm., $F$ is the unique $\alpha$-H\"{o}lder continuous representative of $\langle\nabla \mathsf{d}_p,\nabla\mathsf{d}_q\rangle$ on $B_{r_0}(x_0)$.

\end{proof}
\begin{cor}\label{fioaehfoiaehiof}
For any $\varepsilon>0, D_0\in (0,D)$, there exists $\delta\in (0,\varepsilon)$ such that if $\mathsf{d}(x,p)=\mathsf{d}(x,q)=D_0$, then 
\[
|\mathsf{d}^2(\gamma_{x,p}(t),\gamma_{x,q}(t))/t^2+\langle \nabla \mathsf{d}_p,\nabla\mathsf{d}_q\rangle(x)-2|\leqslant \varepsilon, \, \forall t\in (0,\delta).
\]
\end{cor}

\begin{proof}
Assume the contrary the existence of $\varepsilon_0>0$ and three sequences of points $\{x_i\}$, $\{p_i\}$ and $\{q_i\}$ such that $p_i,q_i\in \partial B_{D_0}(x_i)$ and that for some $s_i\to 0$ it holds
\[
|\mathsf{d}^2(\gamma_{x_i,p_i}(s_i),\gamma_{x_i,q_i}(s_i))/{s_i}^2+\langle \nabla \mathsf{d}_{p_i},\nabla\mathsf{d}_{q_i}\rangle(x_i)-2|\geqslant \varepsilon_0, \ \forall i.
\]

Recall that our volume homogeneity condition yields (\ref{fhoeahofihaefi}) and (\ref{vpxbbctt}). Therefore, applying \cite[Thm. 1.6]{DG18} shows 
\begin{equation}\label{feahfoeahofi}
\sup_{x\in X}\,\mathsf{d}_{\text{pmGH}}\left((X,r^{-1}\mathsf{d},\mathscr{H}^n,x),\left(\mathbb{R}^n,\mathsf{d}_{\mathbb{R}^n},\mathscr{L}^n,0_n\right)\right)\leqslant \Psi(r|K,n,D).
\end{equation}
This implies
\[
(X_i,\mathsf{d}_i,\mathscr{H}^n,x_i):=(X,{s_i}^{-1}\mathsf{d},\mathscr{H}^n,x_i)\xrightarrow{\text{pmGH}}\left(\mathbb{R}^n,\mathsf{d}_{\mathbb{R}^n},\mathscr{L}^n,0_n\right).
\]

By (\ref{fdahfuafhaeiohfo}), Thms. \ref{AAthm} and \ref{AH18}, there exist $\vec{a},\vec{b}\in \mathbb{S}^{n-1}$ such that $\{{s_i}^{-1}(\mathsf{d}_{p_i}-D_0)\}$ and $\{{s_i}^{-1}(\mathsf{d}_{q_i}-D_0)\}$ uniformly converges to linear functions {$(\vec{x}\mapsto \vec{a}\cdot \vec{x})$, $(\vec{x}\mapsto \vec{b}\cdot \vec{x})$} respectively on $B_2(0_n)\subset \mathbb{R}^n$ as $i\to\infty$.

Observe that 
\[
{s_i}^{-1}(\mathsf{d}_{p_i}(\gamma_{x_i,p_i}(s_i t))-D_0)={s_i}^{-1}(\mathsf{d}_{q_i}(\gamma_{x_i,q_i}(s_i t))-D_0)=t,\, \forall t\in [0,2]. 
\]Therefore, as $i\to\infty$, $\gamma_{x_i,p_i}(s_i)\to \vec{a}$ and $\gamma_{x_i,q_i}(s_i)\to \vec{b}$. This implies \[\lim_{i\to\infty}{s_i}^{-1}\mathsf{d}(\gamma_{x_i,p_i}(s_i),\gamma_{x_i,q_i}(s_i))=|\vec{b}-\vec{a}|.
\]

However, Prop. \ref{fehohfaoehf} combined with $H^{1,2}$-convergences of $\{{s_i}^{-1}(\mathsf{d}_{p_i}-D_0)\}$ and $\{{s_i}^{-1}(\mathsf{d}_{q_i}-D_0)\}$ on $B_1(0_n)$ yields a contradiction, since
\[
\lim_{i\to\infty}\langle\nabla \mathsf{d}_{p_i},\nabla \mathsf{d}_{q_i}\rangle(x_i)=\lim_{i\to\infty}\fint_{B_{s_i}(x_i)}\langle\nabla \mathsf{d}_{p_i},\nabla \mathsf{d}_{q_i}\rangle\,\mathrm{d}\mathscr{H}^n=\vec{a}\cdot\vec{b}.
\]
\end{proof}
\begin{prop}[Local bi-Lipschitz coordinates]\label{step4}
For $D_0\in (0,D)$, there exists $r_0=r_0(K,n,D_0)<\min\{D_0/100,(D-D_0)/100\}$ with the following property: given any $x_0\in X$ and $r\in (0,r_0)$,  one can find $n$ points $\{p_i\}\subset \partial B_{D_0}(x_0)$ such that the following holds.
\begin{itemize}
\item 
\[
|\langle\nabla \mathsf{d}_{p_i},\nabla \mathsf{d}_{p_j}\rangle (x_0)-\delta_{ij}|\leqslant \Psi(r|K,n,D).
\]
\item The map
\[
\begin{aligned}
\textbf{U}:B_r(x_0)&\longrightarrow \mathbf{U}(B_r(x_0))\subset\mathbb{R}^n\\
y&\longmapsto(u_1(y),\ldots,u_n(y)),
\end{aligned}
\] 
is bijective, where $u_i:=\mathsf{d}_{p_i}-D_0$. \item $\textbf{U}$, $\textbf{U}^{-1}|_{\textbf{U}(B_r(x_0))}$ are both $(1+\Psi(r|K,n,D))$-Lipschitz continuous.
\end{itemize}
\end{prop}
\begin{proof}
By (\ref{feahfoeahofi}), there exists $r_1=r_1(K,n,D)$ such that 
\[
\sup_{x\in X}\,\mathsf{d}_{\text{pmGH}}\left((X,r^{-1}\mathsf{d},\mathscr{H}^n,x),\left(\mathbb{R}^n,\mathsf{d}_{\mathbb{R}^n},\mathscr{L}^n,0_n\right)\right)\leqslant \Psi(r|K,n,D)<1/1000, \ \forall r<r_1.
\]

For every $x_0\in X$ and every $r<r_1$, Gromov-Hausdorff approximation guarantees the existence of $n$ points $\{q_i\}_{i=1}^n \subset \partial B_r(x_0)$ such that $\mathsf{d}(q_i,q_j)=(\sqrt{2}+\Psi)r$ for any $i\neq j$. Set $p_i:=\gamma_{x_0,q_i}(D_0)$. Cor. \ref{fioaehfoiaehiof} implies $|\langle \nabla \mathsf{d}_{p_i},\nabla\mathsf{d}_{p_j}\rangle(x)-\delta_{ij}|\leqslant \Psi$. 

Since Prop. \ref{fehohfaoehf} guarantees the $\alpha(n)$-H\"{o}lder continuity of the functions $\langle \nabla \mathsf{d}_{p_i},\nabla\mathsf{d}_{p_j}\rangle$ on $B_{r_2}(x_0)$ with $r_2=r_2(K,n,D_0)$, we see
\[
\max_{B_{2r}(x_0)}|\langle \nabla \mathsf{d}_{p_i},\nabla\mathsf{d}_{p_j}\rangle-\delta_{ij}|\leqslant \Psi(r|K,n,D), \ \forall r<r_0:=\min\{r_1,r_2\}.
\]
Finally, the conclusion follows from a combination of H\"{o}lder's inequality and \cite[Thm. 3.4]{H21}.

\end{proof}

We are now in a position to prove the following key proposition.
\begin{prop}\label{step5} Given any fixed point $\bar{x}\in X$, it holds that 
\[
\Delta \,\mathsf{d}_{\bar{x}}(y)=\Delta\, \mathsf{d}_{\bar{x}}(z), \ \forall y,z \in \partial B_R(\bar{x}),\ \forall R\in (0,D).
\]
In other words, $\Delta\, \mathsf{d}_{\bar{x}}$ is radial with respect to $\bar{x}$.
\end{prop}
\begin{proof}
Fix $x_0\in X$ with $\mathsf{d}(\bar{x},x_0)=D_0\in (0,D)$. By Prop. \ref{step4}, there exists $r_0=r_0(K,n,D_0)>0$ and a bi-Lipschitz map $\textbf{U}=(u_1,\ldots,u_n):B_{r_0}(x_0)\rightarrow \mathbb{R}^n$ such that $u_n=\mathsf{d}_{\bar{x}}-D_0$ and $|\langle\nabla u_i,\nabla u_j\rangle-\delta_{ij}|\leqslant 1/10000$ on $B_{r_0}(x_0)$. 

Due to the volume homogeneity assumption, it suffices to show that for $\mathscr{H}^n$-a.e. $x\in B_{r_0}(x_0)$ the following holds.
\begin{equation}\label{aaa4.16}
2\mathscr{H}^n\left(B_t(x)\setminus B_{\mathsf{d}(\bar{x},x)}({}{\bar{x}})\right)-\omega_n t^{n}=2\omega_n t^{n+1}\Delta \,\mathsf{d}_{\bar{x}}(x)+o(t^{n+1}),\ \text{as}\ t\downarrow 0.
\end{equation}
\

\textbf{Step 1 Construction of good coordinate functions}

\

Let us take 
\[
x\in \bigcap_{j,k=1}^n \mathrm{Leb}(\langle \mathrm{Hess} \mathop{u_j},\mathrm{Hess}\mathop{ u_k}\rangle)\cap \bigcap_{j,k,l=1}^n\mathrm{Leb}(\mathrm{Hess}\mathop{u_j}(\nabla u_k,\nabla u_l)) \cap B_{r_0}(x_0),
\]
 and constants $a_{ij}$ ($i,j=1,\ldots,n$) and $\Gamma^i_{jk}$, $b^i_{jk}$ ($i,j,k=1,\ldots,n$) {}{such that 
 \begin{equation}\label{4.18c}
\sum_{k,l=1}^n a_{ik}a_{jl}\langle \nabla u_k,\nabla u_l\rangle(x)=\delta_{ij},\ \Gamma^i_{jk}=\Gamma^i_{kj}=\mathrm{g}_{jl}\mathrm{g}_{km}\mathrm{Hess}\,u_i(\nabla u_l,\nabla u_m),\ -2b^i_{jk}=a_{il}\Gamma^l_{jk},
\end{equation}
where $(\mathrm{g}_{ij})$ is the inverse matrix of $(\langle \nabla u_i,\nabla u_j\rangle(x))$. 

Since 
\[
\lim_{t\to 0}\fint_{B_t(x)} |\mathrm{Hess}\, u_i-\Gamma_{jk}^i \nabla u_j\otimes \nabla u_k|_{\mathsf{HS}}^2\,\mathrm{d}\mathscr{H}^n=0,
\]}the functions $v_i=\sum_{j=1}^n a_{ij}(u_j-u_j(x))+\sum_{j,k=1}^n b^i_{jk}(u_j-u_j(x)) (u_k-u_k(x))$ satisfy the following properties.
 \begin{itemize}
 \item  $\sum_{i=1}^n\left(\| \nabla v_i\|_{L^\infty(B_{r_0}(x))}+\| \Delta v_i\|_{L^\infty(B_{r_0}(x))}\right)\leqslant C$, for some constant $C$ that may vary from line to line.
 \item $v_i(x)=0$ (for $i=1,\ldots,n$), $\langle\nabla v_i,\nabla v_j\rangle(x)=\delta_{ij}$.
 \item \begin{equation}\label{4.17aaaaaa}
\lim_{t\rightarrow 0}\fint_{B_t(x)} |\mathrm{Hess}\mathop{v_i}|_{\mathsf{HS}}^2\,\mathrm{d}\mathscr{H}^n=0,\  \text{for}\ i=1,\ldots,n. 
\end{equation}
 \end{itemize}

{}{Since $\langle\nabla v_i,\nabla v_j\rangle(x)=\delta_{ij}$, using Thm.~\ref{BGineq}, (1,2)-Poincar\'{e} inequality and (\ref{4.17aaaaaa}), and following a standard telescope argument shows that for each $i,j=1,\ldots,n$,
\begin{equation}\label{eqn4.22}
\begin{aligned}
&\fint_{B_t(x)}|\langle\nabla v_i,\nabla v_j\rangle-\delta_{ij}|\mathop{\mathrm{d}\mathscr{H}^n}
\leqslant  \fint_{B_{t}(x)}\left|\langle\nabla v_i,\nabla v_j\rangle-\fint_{B_{t/2}(x)}\langle\nabla v_i,\nabla v_j\rangle\, \mathrm{d}\mathscr{H}^n\right|\mathop{\mathrm{d}\mathscr{H}^n}\\
&+\sum_{l=1}^\infty\left|\fint_{B_{2^{-l}t}(x)}\langle\nabla v_i,\nabla v_j\rangle\, \mathrm{d}\mathscr{H}^n-\fint_{B_{2^{-(l+1)}t}(x)}\langle\nabla v_i,\nabla v_j\rangle\, \mathrm{d}\mathscr{H}^n\right|\mathop{\mathrm{d}\mathscr{H}^n}\\
\leqslant \ &\sum_{l=0}^\infty\fint_{B_{2^{-(l+1)}t}(x)}\left|\langle\nabla v_i,\nabla v_j\rangle-\fint_{B_{2^{-l}t}(x)}\langle\nabla v_i,\nabla v_j\rangle\, \mathrm{d}\mathscr{H}^n\right|\mathop{\mathrm{d}\mathscr{H}^n}\\
\leqslant \ &\sum_{l=0}^\infty\frac{\mathscr{H}^n\left(B_{2^{-l}t}(x)\right)}{\mathscr{H}^n\left(B_{2^{-(l+1)}t}(x)\right)}\fint_{B_{2^{-l}t}(x)}\left|\langle\nabla v_i,\nabla v_j\rangle-\fint_{B_{2^{-l}t}(x)}\langle\nabla v_i,\nabla v_j\rangle\, \mathrm{d}\mathscr{H}^n\right|\mathop{\mathrm{d}\mathscr{H}^n}\\
\leqslant \ &C(K,n)\sum_{l=0}^\infty 2^{-l}t\fint_{B_{2^{-(l+1)}t}(x)}\left({|\mathrm{Hess}\,v_i|_{\mathsf{HS}}}^2+{|\mathrm{Hess}\,v_j|_{\mathsf{HS}}}^2\right)\mathop{\mathrm{d}\mathscr{H}^n}=o(t), \ \ \text{as}\ t\downarrow 0.
\end{aligned}
\end{equation}}

According to the $\alpha(n)$-H\"{o}lder continuity of $\langle\nabla u_i,\nabla u_j\rangle$, one has for each $i,j=1,\ldots,n$, 
\[
|\langle \nabla v_i,\nabla v_j\rangle(y)-\delta_{ij}|=|\langle \nabla v_i,\nabla v_j\rangle(y)-\langle \nabla v_i,\nabla v_j\rangle(x)|\leqslant C(\mathsf{d}(x,y))^\alpha,\ \forall y\in B_{r_0}(x_0).
\]
This implies 
\begin{align}\label{4.3000000}
\sup_{B_{2t}(x)}|\langle \nabla v_i,\nabla v_j\rangle-\delta_{ij}|\leqslant C\,t^\alpha<1/10000,\ \forall t\ll 1.
\end{align}

For convenience, denote by $\textbf{V}=(v_1,\ldots,v_n)$ and by $V(z)$ the matrix $(\langle \nabla v_i,\nabla v_j\rangle(z))_{ij}$ for $z\in B_t(x)$. Applying \cite[Thm.~3.4]{H21} yields that $\mathbf{V}|_{B_t(x)}:B_t(x)\to \mathbf{V}(B_t(x))$ is bijective and the restrictions $\textbf{V}|_{B_{t}(x)}$ and $\textbf{V}^{-1}|_{\textbf{V}(B_{t}(x))}$ are both $(1+\Psi(t|C))$-Lipschitz continuous. By adapting the proof of \cite[Lem.~4.7]{H23}, we verify that 
\begin{equation}\label{fheaohfoahfiahfio}
\textbf{V}_\sharp \left(\mathscr{H}^{n}\llcorner_{B_t(x)}\right)=\left(\det V\circ\mathbf{V}^{-1} \right)^{-1/2}\mathscr{L}^n\llcorner_{\mathbf{V}(B_t(x))}.
\end{equation}

{Let $E$ be an $n\times n$ invertible matrix such that $EV(z)E^T=I_n$. Then the map $\mathbf{W}=(w_1,\ldots,w_n)$ defined by $w_i=\sum_{j}E_{ij}v_j$ ($i=1,\ldots,n$) satisfies
\[
\lim_{y\to z}\frac{{|\mathbf{W}(y)-\mathbf{W}(z)|}^2}{\mathsf{d}^2(y,z)}=1.
\]
Therefore, we obtain
\begin{equation}\label{fiehfioafiohaio}
\left\{
\begin{aligned}
& \limsup_{y\to z}\frac{{|\mathbf{V}(y)-\mathbf{V}(z)|}^2}{{\mathsf{d}}^2(y,z)}=\lambda_{\mathrm{max}}((EE^T)^{-1})=\lambda_{\mathrm{max}}((E^T E)^{-1})=\lambda_{\mathrm{max}}(V(z)),\\
& \liminf_{y\to z}\frac{{|\mathbf{V}(y)-\mathbf{V}(z)|}^2}{\mathsf{d}^2(y,z)}=\lambda_{\mathrm{min}}(V(z)),
\end{aligned}\right.
 \end{equation}
 where $\lambda_{\mathrm{max}}(V(z))$ (resp. $\lambda_{\mathrm{min}}(V(z))$) is the largest (resp. smallest) eigenvalue of $V(z)$. {Moreover, by Prop. \ref{prop5.1} and Corollary \ref{cor5.2} we have the following estimates}:
\begin{equation}\label{4.38}
\frac{\sum_{i} {|\nabla v_i|}^2(z)}{n}\leqslant \lambda_{\max}(V(z))\leqslant 1+\sum_{i,j}|\langle\nabla v_i,\nabla v_j\rangle(z)-\delta_{ij}|,
\end{equation} 
\begin{equation}\label{4.39}
\mathrm{det}(V(z))(\lambda_{\mathrm{max}}(V(z)))^{1-n}\leqslant \lambda_{\min}(V(z))\leqslant \frac{1}{n}\sum_{i=1}^n {|\nabla v_i|}^2(z).
\end{equation}}

\

\ 
\textbf{Step 2 Asymptotic estimates after pulling back to Euclidean space}
\ 

\ 

By (\ref{eqn4.22}) and (\ref{fheaohfoahfiahfio}), we have
\begin{equation}\label{aabb4.25}
\begin{aligned}
&\mathscr{H}^n\left(B_{t}(x)\setminus B_{\mathsf{d}(x,\bar{x})}({}{\bar{x}})\right)\\
=\ &\int_{B_{t}(x)\cap \{u_n\geqslant u_n(x) \}}\left(\det V \right)^{\frac{1}{2}}\mathop{\mathrm{d}\mathscr{H}^n}+\int_{B_{t}(x)\cap \{u_n\geqslant u_n(x) \}}\left(1-\left(\det V \right)^{\frac{1}{2}}\right)\mathop{\mathrm{d}\mathscr{H}^n}\\
=\ &\mathscr{L}^n\Big{(}\textbf{V}\big{(}B_t(x)\cap\{u_n\geqslant u_n(x)\}\big{)}\Big{)}+o(t^{n+1}), \ \ \text{as}\ t\downarrow 0.
\end{aligned}
\end{equation}

We claim
\begin{equation}\label{4.41a}
 \mathscr{L}^n\left(\textbf{V}(B_t(x))\triangle B_t(0_n)\right)=o(t^{n+1}),  \ \ \text{as}\ t\downarrow 0.
 \end{equation}
 
 Set $r:=\sqrt{\sum {v_i}^2}$, $r_+:=\max\{r,\mathsf{d}_x\}$ and $\Omega_t:=\mathbf{V}(\{r_+\leqslant t\})$. For simplicity, we denote by $\tilde{\nabla}, \tilde{\Delta}$ the gradient and Laplacian on $\mathbb{R}^n$, and by $\tilde{f}:=f\circ \mathbf{V}^{-1}$ for $f\in \mathrm{Lip}(B_{4t}(x),\mathsf{d})$. From (\ref{fiehfioafiohaio}) and \cite[Thm.~5.1]{Ch99} we know for $\mathscr{H}^n$-a.e. $z\in B_{2t}(x)$ it holds
{}{\begin{equation}\label{4.42}
 C^{-1}{|\nabla f|}^2(z)\leqslant \frac{{|\nabla f|}^2\left(z\right)}{\lambda_{\max}(V(z))}\leqslant {|\tilde{\nabla} \tilde{f}|}^2(\mathbf{V}(z))\leqslant \frac{{|\nabla f|}^2(z)}{\lambda_{\min}(V(z))}\leqslant C{|\nabla f|}^2(z). 
 \end{equation}}

Let us recall that in{}{\cite[Lem.~3.5, Prop.~4.8]{BMS23}} Bru\'{e}-Mondino-Semola proved
\begin{align}\label{a4.39}
	\int_{0}^t\frac{1}{s^{n+1}}\int_{\partial B_s(x)}r{|\nabla(r-\mathsf{d}_x)|}^2\,\mathrm{d}\mathscr{H}^{n-1}\mathrm{d}s=o(t),\ \text{as}\ t\downarrow 0.
\end{align}
It should be emphasized that the above asymptotic behavior played an essential role in proving \cite[Prop.~4.1]{BMS23}, namely \begin{align}\label{fnaeifhaeiohfeao}
	\mathscr{H}^n(B_t(x))=\omega_n t^n+o(t^{n+1})\  \text{as}\  t\downarrow 0,\  \text{for}\  \mathscr{H}^n\text{-a.e.}\  x\in X.
\end{align}

We first show  
 \begin{equation}\label{fehiaofhoaehfoahfoiahof}
 \mathscr{L}^n(\Omega_t)=\omega_n t^n+o(t^{n+1}), \ \ \text{as}\ t\downarrow 0.
 \end{equation}

Notice that (\ref{eqn4.22}) gives
\begin{equation}\label{fejffaef}
\begin{aligned}
&\int_{B_t(x)} \frac{1}{{\mathsf{d}_x}^{n}}|\langle\nabla v_i,\nabla v_j\rangle-\delta_{ij}|\,\mathrm{d}\mathscr{H}^n=\sum_{l=0}^\infty\int_{A_{2^{-(l+1)}t,2^{-l}t}(x)} \frac{1}{{\mathsf{d}_x}^{n}}|\langle\nabla v_i,\nabla v_j\rangle-\delta_{ij}|\,\mathrm{d}\mathscr{H}^n\\
&\leqslant \sum_{l=0}^\infty \frac{2^{n}}{{(2^{-l}t)}^n}\int_{B_{2^{-l}t}(x)} |\langle\nabla v_i,\nabla v_j\rangle-\delta_{ij}|\,\mathrm{d}\mathscr{H}^n=o(t),\ \text{as}\ t\downarrow 0.
\end{aligned}
\end{equation}

For $t\ll 1$ we may assume $\Omega_t\subset \mathbf{V}(B_{2t}(x))$ and $\Omega_t\subset B_{2t}(0_n)$, because $\mathbf{V}|_{B_{4t}(x)}$ is $(1+\Psi(t|C))$-bi-Lipschitz equivalent with its image. We first estimate that
\[
\begin{aligned}
&\int_{\Omega_t}\frac{\tilde{r}}{(\tilde{r}_+)^{n+1}}{|\tilde{\nabla} (\tilde{r}-\tilde{r}_+)|}^2\,\mathrm{d}\mathscr{L}^n\leqslant \int_{\mathbf{V}(B_{2t}(x))}\frac{\tilde{r}}{(\tilde{r}_+)^{n+1}}{|\tilde{\nabla} (\tilde{r}-\tilde{r}_+)|}^2\,\mathrm{d}\mathscr{L}^n\\
=\ &\int_{\mathbf{V}(B_{2t}(x))}\frac{\tilde{r}}{(\tilde{r}_+)^{n+1}}{|\tilde{\nabla} (\tilde{r}-\tilde{r}_+)|}^2\left(1-\left(\tilde{\det V}\right)^{-\frac{1}{2}}\right)\,\mathrm{d}\mathscr{L}^n\\
&+\int_{\mathbf{V}(B_{2t}(x))}\frac{\tilde{r}}{(\tilde{r}_+)^{n+1}}{|\tilde{\nabla} (\tilde{r}-\tilde{r}_+)|}^2\left(\tilde{\det V}\right)^{-\frac{1}{2}}\,\mathrm{d}\mathscr{L}^n:=I_1+I_2.
\end{aligned}
\]
We remark that $\tilde{r}$ is indeed the distance from the origin. 

On $B_{4t}(x)$, since $|\nabla v_i|\leqslant C$, $|\nabla r|\leqslant C$. Thus $|\nabla (r- r_+)|=\chi_{\{r<\mathsf{d}_x\}}|\nabla(r-\mathsf{d}_x)|\leqslant C$. From (\ref{4.3000000}), (\ref{4.38}), (\ref{4.39}) and (\ref{4.42}) we see $|\tilde{\nabla} (\tilde{r}-\tilde{r}_+)|\leqslant C$. This as well as (\ref{fheaohfoahfiahfio}) and (\ref{fejffaef}) yields
\[
\begin{aligned}
I_1\leqslant \ &C\int_{\mathbf{V}(B_{2t}(x))}\frac{\tilde{r}}{(\tilde{r}_+)^{n+1}}\left|\left(\tilde{\det V}\right)^{-\frac{1}{2}}-1\right|\,\mathrm{d}\mathscr{L}^n=C\int_{B_{2t}(x)}\frac{r}{(r_+)^{n+1}}\left|\left(\det V\right)^{\frac{1}{2}}-1\right|\,\mathrm{d}\mathscr{H}^n\\
\leqslant \ &C\int_{B_{2t}(x)}\frac{1}{{\mathsf{d}_x}^{n}}\left|\left(\det V\right)^{\frac{1}{2}}-1\right|\,\mathrm{d}\mathscr{H}^n=o(t),\ \text{as}\ t\downarrow 0.
\end{aligned}
\]

Applying (\ref{fheaohfoahfiahfio}), (\ref{4.38}), (\ref{4.39}), (\ref{4.42}) and  also implies 
\[
\begin{aligned}
I_2\leqslant \ &C\int_{\mathbf{V}(B_{2t}(x))}\frac{\tilde{r}}{(\tilde{r}_+)^{n+1}}{|\nabla (r-r_+)|}^2\circ \mathbf{V}^{-1}\left(\tilde{\det V}\right)^{-\frac{1}{2}}\,\mathrm{d}\mathscr{L}^n\\
= \ &C\int_{B_{2t}(x)}\frac{r}{(r_+)^{n+1}}{|\nabla (r-r_+)|}^2\,\mathrm{d}\mathscr{H}^n=C\int_{B_{2t}(x)\cap\{r<\mathsf{d}_x\}}\frac{r}{(r_+)^{n+1}}{|\nabla (r-\mathsf{d}_x)|}^2\,\mathrm{d}\mathscr{H}^n\\
\leqslant \ &C\int_{B_{2t}(x)}\frac{r}{{\mathsf{d}_x}^{n+1}}{|\nabla (r-\mathsf{d}_x)|}^2\,\mathrm{d}\mathscr{H}^n=o(t),\ \text{as}\ t\downarrow 0,
\end{aligned}
\]
where the last equality follows from a combination of (\ref{a4.39}) and co-area formula. Therefore we know
\begin{align}\label{faifaehoifjeaijfaeio}
\int_{\Omega_t}\frac{\tilde{r}}{(\tilde{r}_+)^{n+1}}{|\tilde{\nabla} (\tilde{r}-\tilde{r}_+)|}^2\,\mathrm{d}\mathscr{L}^n=o(t),\ \text{as}\ t\downarrow 0.
\end{align}

By (\ref{4.42}), for $\mathscr{L}^n$-a.e. $ \tilde{z}\in B_{2t}(0_n)$ it holds
\[
\begin{aligned}
\left|{|\tilde{\nabla}\tilde{\mathsf{d}_x}|}^2-1\right|(\tilde{z})\ &\leqslant C\max\left\{\left|\lambda_{\max}(V\circ \mathbf{V}^{-1}(\tilde{z}))-1\right|,\left|\lambda_{\min}(V\circ \mathbf{V}^{-1}(\tilde{z}))-1\right|\right\}\\
&\leqslant C\left(\left|\lambda_{\max}(V\circ \mathbf{V}^{-1}(\tilde{z}))-1\right|+\left|\lambda_{\min}(V\circ \mathbf{V}^{-1}(\tilde{z}))-1\right|\right)
\end{aligned}.
\]
Then similarly by (\ref{fheaohfoahfiahfio}), (\ref{4.38}), (\ref{4.39}) and (\ref{fejffaef}) we have
\begin{align}\label{feiaofhoaihfiae}
\int_{B_{2t}(0_n)}\frac{\tilde{r}}{(\tilde{r}_+)^{n+1}}\left|{|\tilde{\nabla}\tilde{\mathsf{d}_x}|}^2-1\right|\,\mathrm{d}\mathscr{L}^n+\frac{1}{t^n}\int_{B_{2t}(0_n)}\left|{|\tilde{\nabla}\tilde{\mathsf{d}_x}|}^2-1\right|\,\mathrm{d}\mathscr{L}^n=o(t), \ \  \text{as}\ t\downarrow 0.
\end{align}

From co-area formula and Fubini's Thm. we obtain
\[
\begin{aligned}
&\int_{B_{2t}(0_n)}\frac{\tilde{r}_+-\tilde{r}}{(\tilde{r}_+)^{n+1}}\,\mathrm{d}\mathscr{L}^n\leqslant \int_{B_{2t}(0_n)}\frac{\tilde{r}_+-\tilde{r}}{{\tilde{r}}^{n+1}}\,\mathrm{d}\mathscr{L}^n\\
\leqslant\ &\int_0^{2t}\int_ {\mathbb{S}^{n-1}}s^{-2}\int_{0}^s\left(|\tilde{\nabla} \tilde{r}_+|(\tau\sigma)-1\right)\,\mathrm{d}\tau\mathrm{d}\mathrm{vol}_{{S}^{n-1}}(\sigma)\mathrm{d}s\\
\leqslant\ & \int_0^{2t}\int_ {\mathbb{S}^{n-1}}s^{-2}\int_{0}^s\left||\tilde{\nabla} \tilde{r}_+|-1\right|\,\left(|\tilde{\nabla} \tilde{r}_+|+1\right)(\tau\sigma)\,\mathrm{d}\tau\mathrm{d}\mathrm{vol}_{{S}^{n-1}}(\sigma)\mathrm{d}s\\
=\ & \int_0^{2t}s^{-2}\int_{0}^s \tau^{1-n}\int_{\partial B_\tau(0_n)}\left|{|\tilde{\nabla} \tilde{r}_+|}^2-1\right|\,\mathrm{d}\mathscr{H}^{n-1}\mathrm{d}\tau\mathrm{d}s,
\end{aligned}
\]
while applying integral by parts and (\ref{feiaofhoaihfiae}), recalling the fact that $|\tilde{\nabla} \tilde{r}_+|=\chi_{\{\tilde{r}\geqslant \tilde{\mathsf{d}_x}\}}+\chi_{\{\tilde{r}<\tilde{\mathsf{d}_x}\}}|\tilde{\nabla}\tilde{\mathsf{d}_x}|$ to the last term give \begin{equation}\label{4.46}
\begin{aligned}
&\int_{B_{2t}(0_n)}\frac{\tilde{r}_+-\tilde{r}}{(\tilde{r}_+)^{n+1}}\,\mathrm{d}\mathscr{L}^n
\leqslant \int_0^{2t}\frac{1}{s^{n+1}} \int_{ B_s(0_n)}\left|{|\tilde{\nabla} \tilde{r}_+|}^2-1\right|\,\mathrm{d}\mathscr{L}^{n}\,\mathrm{d}s\\
&+(n-1)\,\omega_n \int_0^{2t}\frac{1}{s^2}\int_0^s\fint_{B_\tau(0_n)}\,\left|{|\tilde{\nabla} \tilde{r}_+|}^2-1\right|\,\mathrm{d}\mathscr{L}^{n}\mathrm{d}\tau\,\mathrm{d}s=o(t),\ \ \text{as}\ t\downarrow 0.
\end{aligned}
\end{equation}

Let
\[
\begin{aligned}
&\int_{\Omega_t}\frac{\tilde{r}}{(\tilde{r}_+)^{n+1}}{|\tilde{\nabla} (\tilde{r}_+-\tilde{r})|}^2\,\mathrm{d}\mathscr{L}^n\\
=\ &\int_{\Omega_t}\frac{\tilde{r}}{(\tilde{r}_+)^{n+1}}(1+{|\tilde{\nabla} \tilde{r}_+|}^2)\,\mathrm{d}\mathscr{L}^n-2\int_{\Omega_t}\frac{\tilde{r}}{(\tilde{r}_+)^{n+1}}\langle\tilde{\nabla} \tilde{r},\tilde{\nabla} \tilde{r}_+\rangle\,\mathrm{d}\mathscr{L}^n:=I_3+I_4.
\end{aligned}
\]
Then (\ref{feiaofhoaihfiae}) and (\ref{4.46}) show
\begin{equation}\label{feahifhaeoifhiao}
\begin{aligned}
I_3=\ &2\int_{\Omega_t\cap\{\tilde{r}>\tilde{\mathsf{d}_x}\}}\frac{\tilde{r}}{(\tilde{r}_+)^{n+1}}\,\mathrm{d}\mathscr{L}^n+\int_{\Omega_t\cap\{\tilde{r}\leqslant \tilde{\mathsf{d}_x}\}}\frac{\tilde{r}}{(\tilde{r}_+)^{n+1}}\left(1+{|\tilde{\nabla} \tilde{\mathsf{d}_x}|}^2\right)\,\mathrm{d}\mathscr{L}^n\\
=\ &2\int_{\Omega_t}\frac{\tilde{r}}{(\tilde{r}_+)^{n+1}}\,\mathrm{d}\mathscr{L}^n+\int_{\Omega_t\cap\{\tilde{r}\leqslant \tilde{\mathsf{d}_x}\}}\frac{\tilde{r}}{(\tilde{r}_+)^{n+1}}\left({|\tilde{\nabla} \tilde{\mathsf{d}_x}|}^2-1\right)\,\mathrm{d}\mathscr{L}^n\\
=\  &2\int_{\Omega_t}\frac{1}{(\tilde{r}_+)^{n}}\,\mathrm{d}\mathscr{L}^n+2\int_{\Omega_t}\frac{\tilde{r}-\tilde{r}_+}{(\tilde{r}_+)^{n+1}}\,\mathrm{d}\mathscr{L}^n+\int_{\Omega_t\cap\{\tilde{r}\leqslant \tilde{\mathsf{d}_x}\}}\frac{\tilde{r}}{(\tilde{r}_+)^{n+1}}\left({|\tilde{\nabla} \tilde{\mathsf{d}_x}|}^2-1\right)\,\mathrm{d}\mathscr{L}^n\\
=\ &2\int_{\Omega_t}\frac{1}{(\tilde{r}_+)^{n}}\,\mathrm{d}\mathscr{L}^n+o(t)=2\int_0^t \frac{\mathscr{H}^{n-1}(\partial \Omega_s)}{s^n}\,\mathrm{d}s+o(t),\ \text{as}\ t\downarrow 0,
\end{aligned}
\end{equation}
where the last equality uses co-area formula. 

{}{As for $I_4$, first fix $s\in (0,t)$ and let $\{\varphi_i\}$ be a sequence of smooth functions on $B_{2t}(0_n)$ which uniformly converges to $\tilde{r}_+$ on $\Omega_t\setminus \Omega_s$ with $\|\varphi_i-\tilde{r}_+\|_{H^{1,2}(\Omega_t\setminus \bar{\Omega}_s)}\to 0$. This implies
\[
\int_{\Omega_t\setminus \Omega_s}\frac{1}{(\tilde{r}_+)^{n+1}}\langle\tilde{\nabla} {\tilde{r}}^2,\tilde{\nabla} \tilde{r}_+\rangle\,\mathrm{d}\mathscr{L}^n=\lim_{i\to\infty} \int_{\Omega_t^i\setminus \Omega_s^i}\frac{1}{{\varphi_i}^{n+1}}\langle\tilde{\nabla} {\tilde{r}}^2,\tilde{\nabla} \varphi_i\rangle\,\mathrm{d}\mathscr{L}^n,
\]
where $\Omega_t^i:=\{\varphi_i\leqslant t\}$. Sard's Thm. guarantees for $\mathscr{L}^1$-a.e. $\mu \in (s,t)$, $\partial \Omega^i_\mu$ is a smooth submanifold. Thus a combination of co-area formula and Gauss-Green formula then gives
\[
\begin{aligned}
&\int_{\Omega_t^i\setminus \Omega_s^i}\frac{1}{{\varphi_i}^{n+1}}\langle\tilde{\nabla} {\tilde{r}}^2,\tilde{\nabla} \varphi_i\rangle\,\mathrm{d}\mathscr{L}^n=\int_s^t \frac{1}{\mu^{n+1}}\int_{\partial \Omega_\mu^i} \left\langle\tilde{\nabla} {\tilde{r}}^2,\frac{\tilde{\nabla} \varphi_i}{|\tilde{\nabla} \varphi_i|}\right\rangle\,\mathrm{d}\mathscr{H}^{n-1}\mathrm{d}\mu\\
=\ &\int_s^t \frac{1}{\mu^{n+1}}\int_{ \Omega_\mu^i} \Delta {\tilde{r}}^2\,\mathrm{d}\mathscr{L}^{n}\mathrm{d}\mu=2\int_s^t \frac{n\mathscr{L}^{n}\left( \Omega_\mu^i\right)}{\mu^{n+1}} \mathrm{d}\mu.
\end{aligned}
\]

Now letting first $i\to\infty$ then $s\to 0$ we know
\begin{equation}\label{4.47a}
\begin{aligned}
&2\int_{\Omega_t}\frac{\tilde{r}}{(\tilde{r}_+)^{n+1}}\langle\tilde{\nabla} \tilde{r},\tilde{\nabla} \tilde{r}_+\rangle\,\mathrm{d}\mathscr{L}^n=\int_{\Omega_t}\frac{1}{(\tilde{r}_+)^{n+1}}\langle\tilde{\nabla} {\tilde{r}}^2,\tilde{\nabla} \tilde{r}_+\rangle\,\mathrm{d}\mathscr{L}^n=2\int_0^t \frac{n\mathscr{L}^n(\Omega_s)}{s^{n+1}}\,\mathrm{d}s.
\end{aligned}
\end{equation}}

Since $B_{(1-\Psi(t|C))t}(0_n)\subset \Omega_t\subset B_{(1+\Psi(t|C))t}(0_n)$, we have $\mathscr{L}^n(\Omega_t)=\omega_n t^n+o(t^n)$ as $t\downarrow 0$. Therefore by (\ref{faifaehoifjeaijfaeio}), (\ref{feahifhaeoifhiao}) and (\ref{4.47a}), we deduce
\[
\frac{\mathscr{L}^n(\Omega_t)}{t^n}-\omega_n=\int_0^t \left(\frac{\mathscr{H}^{n-1}(\partial \Omega_s)}{s^n}- \frac{n\mathscr{L}^n(\Omega_s)}{s^{n+1}}\right)\,\mathrm{d}s=o(t) \ \ \text{as}\ t\downarrow 0,
\] 
proving (\ref{fehiaofhoaehfoahfoiahof}).

By (\ref{fnaeifhaeiohfeao}) and our volume homogeneity assumption, $\mathscr{H}^n(B_t(x))=\omega_n t^n+o(t^{n+1})$ as $t\downarrow 0$. Combining (\ref{eqn4.22}) and (\ref{fheaohfoahfiahfio}) then shows $\mathscr{L}^n(\mathbf{V}(B_t(x)))=\omega_n t^n+o(t^{n+1})$ as $t\downarrow 0$, which together with (\ref{fehiaofhoaehfoahfoiahof}) proves (\ref{4.41a}).
\ 

\

\textbf{Step 3 Calculation on Euclidean space}

\

By (\ref{aabb4.25}) and (\ref{4.41a}), it suffices to check 
\begin{equation}\label{a4.26}
{}{2\mathscr{L}^n\Big{(}B_t(0_n)\cap \mathbf{V}\big{(}\{u_n\geqslant u_n(x)\}\big{)}\Big{)}-\omega_n t^n}=2\,\omega_n\,  \Delta\, \mathsf{d}_{\bar{x}}(x)(t^{n+1}+o(t^{n+1})).
\end{equation}

For simplicity denote by {}{$f=u_n\circ \mathbf{V}^{-1}$}, $c=f(0_n)$, $A=(a_{ij})$ and $A^{-1}=({a}^{ij})$. We also denote by {}{$\mathrm{g}^{ij}(\tilde{z})=\langle\nabla u_i ,\nabla u_j\rangle\circ\mathbf{V}^{-1}(\tilde{z})$, $G(\tilde{z})=(\mathrm{g}^{ij}(\tilde{z}))$ and $\mathrm{g}_{ij}(\tilde{z})=(G^{-1}(\tilde{z}))_{ij}$ for $\tilde{z}\in B_{2t}(0_n)$}. Due to \cite{HT03}, {}{the limit of $t^{-(n+1)}/(2\omega_n)$ multiples the left-hand side of (\ref{a4.26})  as $t\to 0$} corresponds precisely to the mean curvature of the level set  $\{f=c\}$ at the origin. This mean curvature can be expressed as follows (with all partial derivatives evaluated at the origin).
\begin{equation}\label{e4.29}
\begin{aligned}
\frac{\Delta^{\mathbb{R}^n} f}{\left|\nabla^{\mathbb{R}^n} f\right|}-\frac{\mathrm{Hess}^{\mathbb{R}^n}\mathop{f}\left(\nabla^{\mathbb{R}^n} f,\nabla^{\mathbb{R}^n} f\right)}{\left|\nabla^{\mathbb{R}^n} f\right|^{3}}
=\left(\sum_i {f_i}^2\right)^{-\frac{1}{2}}\sum_i f_{ii}-\left(\sum_i {f_i}^2\right)^{-\frac{3}{2}}\sum_{i,j}f_{ij}f_i f_j.
\end{aligned}
\end{equation}

{}{Define the $(i,j)$-entry of the Jacobi matrix as $J_{ij}(\tilde{z})=\partial v_i/\partial u_j(\mathbf{V}^{-1}(\tilde{z}))$,  and set $J^{ij}(\tilde{z})=(J^{-1}(\tilde{z}))_{ij}$ for every $\tilde{z}\in B_{2t}(0_n)$. By the definition of $\{v_i\}_{i=1}^n$,
\[
J_{ij}(\tilde{z})=a_{ij}+2\sum_{k=1}^n b^{i}_{jk}\big{(}u_k\circ\mathbf{V}^{-1}(\tilde{z})-u_k(x)\big{)}.
\]
Thus 
\begin{align}\label{feahfhaeio}
f_i(0_n)=J^{ni}(0_n)=a^{ni},\ \ \text{for}\ i=1,\ldots,n.
\end{align}
 
 To calculate $f_{ij}(0_n)=\partial J^{ni}/\partial v_j (0_n)$, notice that by chain rule we get 
\[
\begin{aligned}
\frac{\partial }{\partial v_j} J^{ni} =\frac{\partial u_m}{\partial v_j} \frac{\partial }{\partial u_m}  J^{ni} =J^{mj}\frac{\partial }{\partial u_m}  J^{ni}=-J^{mj}\left(J^{n\alpha}J^{\beta i}\frac{\partial }{\partial u_m}  J_{\alpha\beta}\right),
\end{aligned}
\]
which implies
\begin{align}\label{4.50}
f_{ij}(0_n)=-2a^{mj}a^{n\alpha}a^{\beta i}b^\alpha_{\beta m}.
\end{align}}

By (\ref{4.18c}), {}{$a_{ik}\mathrm{g}^{kl}(0_n)a_{jl}=\delta_{ij}$}, which yields 
\begin{align}\label{4.51}
\mathrm{g}^{ij}(0_n)=a^{ik}a^{jk}.
\end{align}
 Therefore $|\tilde{\nabla } f|(0_n)=\sum_i (a^{ni})^2=\mathrm{g}^{nn}(0_n)=1$. 
 
 Recall in \cite{H18}, Han proved that $\Delta v_i$ coincides $\mathscr{H}^n$-a.e. with $\langle \mathrm{Hess} \,v_i ,\mathrm{g}\rangle$. Applying (\ref{eqn4.22}) and H\"{o}lder's inequality then shows
\begin{align}\label{5.51}
0=\lim_{t\rightarrow 0}\fint_{B_t(x)}\Delta v_i\mathop{\mathrm{d}\mathscr{H}^n}= a_{ij}\Delta u_j(x)+2 \mathrm{g}^{jk}(0_n)b^{i}_{jk}, \ \text{for}\  i=1,\ldots,n.
\end{align}

Since the matrix $(\mathrm{g}^{ij})$ is symmetric, from (\ref{4.50}), (\ref{4.51}) and (\ref{5.51}) we see
\[
\sum_{i} f_{ii}(0_n)=-2\sum_{i,\alpha,\beta,m} a^{mi}a^{n\alpha}a^{\beta i}b^\alpha_{\beta m}=-2\sum_{\alpha,\beta,m}\mathrm{g}^{m\beta}(0_n)a^{n\alpha}b^\alpha_{\beta m}=\Delta u_n(x).
\]
Therefore we conclude by computing the Hessian term in (\ref{e4.29}) as follows. \[
\begin{aligned}
&\sum_{i,j} f_{ij}(0_n)f_i(0_n) f_j(0_n)=-2\sum_{i,j,m,\alpha,\beta}\left(a^{m j}a^{n\alpha}a^{\beta i}b^\alpha_{\beta m}\right)(a^{ni}a^{nj})\ (\text{by} \ (\ref{feahfhaeio}), (\ref{4.50}))\\
&=-2\sum_{m,\alpha,\beta} a^{n\alpha}\mathrm{g}^{mn}(0_n)\mathrm{g}^{n\beta}(0_n)b^\alpha_{\beta m}\  (\text{by}\  (\ref{4.51}))\\
&=\sum_{\alpha,l}a^{n\alpha}\left(a_{\alpha l}\mathrm{Hess}\,u_l(\nabla u_n,\nabla u_n)(x)\right)\ (\text{by} \ (\ref{4.18c}))\\
&=\mathrm{Hess}\,u_n(\nabla u_n,\nabla u_n)(x)=0.
\end{aligned} 
\]
\end{proof}

{}{\begin{remark}
By combining (\ref{eqn4.6cccc}), (\ref{deltadistance}), (\ref{aaa4.16}), Cor. \ref{step3}, Prop. \ref{step5} and our volume homogeneity assumption, we deduce the existence of a function $\tau\in L^\infty_{\mathrm{loc}}((0,D))$ with the following properties.
\begin{itemize}
\item For any $r\in (0,D)$, $\tau(r)$ has the following two-sided bound:
\begin{align}\label{fehaiohfeoi}
-\frac{ C(K,n,D)}{D-r}\leqslant -\frac{V_{K,n}'(D-r)}{V_{K,n}(D-r)}\leqslant \frac{\tau(r)}{n-1}\leqslant \frac{V_{K,n}'(r)}{V_{K,n}(r)}\leqslant\frac{C(K,n,D)}{r}.
\end{align}
\item For any $x_0\in X$, by letting $Q_{x_0}=\partial B_{D/3}(x_0)$, there exists a bijection from $B_{D}(x_0)\setminus \{x_0\}$ to $Q_{x_0}\times (0,D)$ given by $x\mapsto (\gamma_{x_0,x}(D/3),\mathsf{d}(x,x_0))$. Moreover, the disintegration of measure with respect to $\mathsf{d}_{x_0}$ takes the form
\[
\mathrm{d}\mathscr{H}^n=\sigma_{x_0}(\beta,r)\,\mathrm{d}r\mathrm{d}\mathfrak{q}_{x_0}(\beta),
\]
where the density function $\sigma_{x_0}$ satisfies the multiplicative relation:
\begin{equation}\label{4.55}
\sigma_{x_0}(\beta,r_2)=\sigma_{x_0}(\beta,r_1)\exp\left(\int_{r_1}^{r_2}\tau(r)\,\mathrm{d}r\right),\ \forall \, r_1,r_2\in (0,D),\forall \beta \in Q_{x_0}.
\end{equation}
\end{itemize}
\end{remark}}
Indeed, this reveals the radial dependence of the measure.
\begin{prop}\label{prop4.23}
{}{There exists a locally Lipschitz function $\omega:(0,D)\to (0,\infty)$ with the following property: for any $x_0\in X$, under the disintegration with respect to $\mathsf{d}_{x_0}$, for any $r\in (0,D)$ it holds $\sigma_{x_0}(\cdot,r)=\omega(r)$. Consequently, Rmk. \ref{rmk4.16} implies the $(n-1)$-dimensional Hausdorff measure on geodesic spheres can be expressed as
\begin{align}\label{4.57}
\mathrm{d}\mathscr{H}^{n-1}\llcorner_{\partial B_r(x_0)}=\omega(r)\,\mathrm{d}\mathfrak{q}_{x_0},
\end{align}
where $\mathfrak{q}_{x_0}$ is the pushforward measure defined in Cor. \ref{step3}.}
\end{prop}
\begin{proof}
According to (\ref{4.55}), it suffices to show $\sigma(\cdot,D/3)$ is a constant function on $Q_{x_0}$. For $x\in \partial B_{D/3}(x_0)$, let $\underline{B}_r(x):=\partial B_{D/3}(x_0)\cap B_r(x)$. By Prop. \ref{step4}, when $r\in (0,r_0(K,n,D))$, there exists a bi-Lipschitz coordinate function $\mathbf{U}=(u_1,\ldots,u_n):B_r(x)\to\mathbb{R}^n$ such that $u_n=\mathsf{d}_{x_0}-D/3$. Because the bi-Lipschitz constant is $1+\Psi(r|K,n,D)$, we have
\[
\begin{aligned}
&\mathscr{H}^{n-1}(\underline{B}_r(x))\leqslant (1+\Psi)^n\mathscr{H}^{n-1}\Big{(}\mathbf{U}\big{(}\underline{B}_r(x)\big{)}\Big{)}\\
\leqslant\ &\mathscr{H}^{n-1}\Big{(}\mathbf{U}(\{u_n=0\})\cap B_{(1+\Psi)r}(0_n)\Big{)}(1+\Psi)^n=\omega_{n-1}r^{n-1}(1+\Psi)^{2n},
\end{aligned}
\]
and similarly, 
\[
\mathscr{H}^{n-1}(\underline{B}_r(x))\geqslant \omega_{n-1}r^{n-1}(1-\Psi)^{2n}.
\]
Thus, for $r<r_0/3$ we obtain
\begin{align}\label{ttkk111}
\sup_{x\in \partial B_{D/3}(x_0)}\frac{\mathscr{H}^{n-1}(\underline{B}_{2r}(x))}{\mathscr{H}^{n-1}(\underline{B}_{r}(x))}\leqslant C(K,n,D).
\end{align}

Recall from Sec. \ref{sec4.1} and Cor. \ref{step3} that $\mathfrak{Q}_{x_0}$ is the quotient map from $B_{D}(x_0)\setminus \{x_0\}$ to $Q_{x_0}=\partial B_{D/3}(x_0)$ defined by $x\mapsto \gamma_{x_0,x}(D/3)$, and $\mathfrak{q}_{x_0}$ is the push forward measure $(\mathfrak{Q}_{x_0})_\sharp ((\mathscr{H}^n(X))^{-1}\mathscr{H}^n)$. Thus $\mathfrak{Q}_{x_0}^{-1}(\underline{B}_r(p))=\underline{B}_r(p)\times (0,D)$, which implies
\[
\left(\mathscr{H}^n(X)\right)\mathfrak{q}_{x_0}(\underline{B}_r(p))=\mathscr{H}^n\left(\mathfrak{Q}_{x_0}^{-1}(\underline{B}_r(p))\right)=  \int_{\underline{B}_r(p)} \int_0^D\sigma_{x_0}(\beta,t)\,\mathrm{d}t\mathrm{d}\mathfrak{q}_{x_0}(\beta).
\]

From (\ref{fehaiohfeoi}), (\ref{4.55}) and Rmk. \ref{rmk4.16} we know
\[
\begin{aligned}
&\mathscr{H}^n\left(\mathfrak{Q}_{x_0}^{-1}(\underline{B}_r(p))\right)=\int_{\underline{B}_r(p)} \sigma_{x_0}(\beta,D/3)\,\mathrm{d}\mathfrak{q}_{x_0}(\beta)\int_0^D \exp\left(\int_{D/3}^t\tau(s)\,\mathrm{d}s\right)\,\mathrm{d}t\\
=\ &\mathscr{H}^{n-1}(\underline{B}_r(p))\int_0^D \exp\left(\int^t_{D/3}\tau(s)\,\mathrm{d}s\right)\,\mathrm{d}t.
\end{aligned}
\]Combining this with (\ref{ttkk111}), we conclude that the metric measure space $(Q_{x_0},\mathsf{d},\mathfrak{q}_{x_0})$ satisfy the local doubling property. Additionally, for any $\beta\in Q_{x_0}$ and $r\in (0,r_0/3)$ it holds
\[
\fint_{\underline{B}_r(p)}\sigma_{x_0}(\beta,D/3)\,\mathrm{d}\mathfrak{q}_{x_0}(\beta)=\mathscr{H}^n(X) \left(\int_0^D \exp\left(\int^t_{D/3}\tau(s)\,\mathrm{d}s\right)\,\mathrm{d}t\right)^{-1}.
\]
Finally, the conclusion follows from the Lebesgue differentiation theorem.

\end{proof}
\begin{proof}[Proof of the second statement of Thm. \ref{thm5.1}: \textbf{smoothness}]\ 

For every point $x\in X$, define the operators
\[
\begin{aligned}
{}{A_x}:\mathrm{Lip}(X,\mathsf{d})&\longrightarrow C\big{(}(0,D)\big{)}\\
f&\longmapsto \left(r\mapsto\fint_{\partial B_r(x)}f\,\mathrm{d}\mathscr{H}^{n-1}\right),
\end{aligned}
\]
and
\[
\begin{aligned}
{}{R_x}:\mathrm{Lip}\big{(}(0,D)\big{)}&\longrightarrow \mathrm{Lip}(X,\mathsf{d})\\
f&\longmapsto \left(y\mapsto f(\mathsf{d}(x,y))\right).
\end{aligned}
\]

For $f\in \mathrm{Test}F(X,\mathsf{d},\mathscr{H}^n)$ and $x\in X$, the disintegration with respect to $\mathsf{d}_x$ allows us to identify $y\in B_{D}(x)\setminus \{x\}$ as $(\gamma_{x,y}(D/3),\mathsf{d}(x,y))$. Since $\mathfrak{q}_x(Q_{x})=1$, applying Prop. \ref{prop4.23} yields
\[
 R_x A_x  f(y)= \fint_{{}{\partial }B_{\mathsf{d}(x,y)}(x)}  f \mathop{\mathrm{d}\mathscr{H}^{n-1}} = \int_{Q_x} f(\cdot,\mathsf{d}(x,y))\,\mathrm{d}\mathfrak{q}_x, \ \forall y\in B_D(x)\setminus \{x\}.
\]
This implies $R_x A_x f\in \mathrm{Lip}(X,\mathsf{d})$ with $\mathrm{Lip}(R_xA_x f)\leqslant \mathrm{Lip}(f)$. {}{Observe that for any $\zeta\in \mathrm{Lip}_c((0,D))$ we have
\[
\begin{aligned}
&\int_0^D \zeta' A_xf\,\mathrm{d}r=\int_0^D \int_{Q_x}\zeta'f\,\mathrm{d}\mathfrak{q}_x\mathrm{d}r\ (\text{by Prop.}\ \ref{prop4.23})\\
=\ &-\int_{Q_x}\int_0^D\zeta\frac{\partial}{\partial r}f\,\mathrm{d}r\mathrm{d}\mathfrak{q}_x\ (\text{Fubini's Thm. and integral by parts})\\
=\ &-\int_0^D \zeta \int_{Q_x}\frac{\partial}{\partial r}f\,\mathrm{d}\mathfrak{q}_x\mathrm{d}r \ (\text{Fubini's Thm.}).
\end{aligned}
\]

If in addition $\zeta\in C_c^2((0,D))$, then we obtain\[
\begin{aligned}
&-\int_0^D \zeta'\int_{\partial B_r(x)}\langle\nabla f,\nabla \mathsf{d}_x\rangle\,\mathrm{d}\mathscr{H}^{n-1}\,\mathrm{d}r=-\int_0^D \int_{\partial B_r(x)}\langle\nabla f,\nabla \zeta\circ\mathsf{d}_x\rangle\,\mathrm{d}\mathscr{H}^{n-1}\,\mathrm{d}r\\
&=-\int_X\langle\nabla f,\nabla \zeta\circ\mathsf{d}_x\rangle\,\mathrm{d}\mathscr{H}^{n} \ (\text{co-area formula})=\int_X  f\Delta (\zeta\circ\mathsf{d}_x)\,\mathrm{d}\mathscr{H}^{n}\\
&=\int_X  f\big{(}\zeta'\circ\mathsf{d}_x\Delta \mathsf{d}_x+\zeta''\circ\mathsf{d}_x\big{)}\,\mathrm{d}\mathscr{H}^{n}\  (\text{by \cite[Prop. 5.2.3]{G18a}})\\
&=\int_{Q_x}\int_0^D f\left(\zeta'\frac{\partial \log\sigma_{x}}{\partial r}+\xi''\right)\sigma_x\,\mathrm{d}r\mathrm{d}\mathfrak{q}_x\  (\text{by}\ (\ref{deltadistance}))\\
&=\int_{Q_x}\int_0^D f\frac{\partial (\zeta' \sigma_{x})}{\partial r}\,\mathrm{d}r\mathrm{d}\mathfrak{q}_x=-\int_{Q_x}\int_0^D \frac{\partial f}{\partial r}\zeta'\sigma_x\,\mathrm{d}r\mathrm{d}\mathfrak{q}_x\  (\text{integral by parts})\\
&=-\int_0^D\zeta'\omega\int_{Q_x} \frac{\partial f}{\partial r}\,\mathrm{d}r\mathrm{d}\mathfrak{q}_x\   (\text{Prop. \ref{prop4.23} and Fubini's Thm.}).
\end{aligned}
\]
Thus we obtain for $\mathscr{L}^1$-a.e. $r\in (0,D)$, 
\begin{align}\label{a4.62}
\omega(r)(A_x f)'(r)= \omega(r)\int_{Q_x}\frac{\partial f}{\partial r} (\cdot,r)\,\mathrm{d}\mathfrak{q}_x=\int_{\partial B_r(x)}\langle\nabla f,\nabla \mathsf{d}_x\rangle\,\mathrm{d}\mathscr{H}^{n-1}.
\end{align}}

{}{From Rmk. \ref{GaussGreen} we know
\begin{align}\label{a4.63}
\int_{\partial B_r(x)}\langle\nabla f,\nabla \mathsf{d}_x\rangle\,\mathrm{d}\mathscr{H}^{n-1}=\int_{B_r(x)}\Delta f\,\mathrm{d}\mathscr{H}^n,\ \mathscr{L}^1\text{-a.e.}\ \  r\in (0,D),
\end{align}
which together with co-area formula also yields for $\mathscr{L}^1$-a.e. $r\in (0,D)$,
\begin{align}\label{fe4.64}
\frac{\partial}{\partial r}\left(\int_{B_r(x)} \Delta f\mathop{\mathrm{d}\mathscr{H}^{n}}\right)=\int_{\partial B_r(x)} \Delta f\,\mathrm{d}\mathscr{H}^{n-1}.
\end{align}}

Therefore, given any $\varphi\in \mathrm{Test}F(X,\mathsf{d},\mathscr{H}^n)$, we compute\[
\begin{aligned}
&\int_X \Delta \varphi \, R_x A_x  f\mathop{\mathrm{d}\mathscr{H}^{n}}=\int_X \langle\nabla \varphi,\nabla \mathsf{d}_x\rangle (A_x  f)'(\mathsf{d}_x)\mathop{\mathrm{d}\mathscr{H}^{n}}\\
=&\int_0^D  (A_x  f)'(r)\int_{\partial B_r(x)}\langle\nabla \varphi,\nabla \mathsf{d}_x\rangle\,\mathrm{d}\mathscr{H}^{n-1}\mathrm{d}r\ (\text{co-area formula})\\
=&\int_0^D \omega(r)\left[\int_{Q_x}\frac{\partial}{\partial r} f(\cdot,r)\mathop{\mathrm{d}\mathfrak{q}_x} \int_{Q_x}\frac{\partial}{\partial r}\varphi(\cdot,r)\mathop{\mathrm{d}\mathfrak{q}_x}\right]\mathop{\mathrm{d}r}\ (\text{by}\  (\ref{a4.62}))\\
=&\int_0^D \left[\int_{\partial B_r(x)}\langle\nabla f,\nabla \mathsf{d}_x\rangle\mathop{\mathrm{d}\mathscr{H}^{n-1}} \frac{\partial}{\partial r}\left(\int_{Q_x}\varphi(\cdot,r)\mathop{\mathrm{d}\mathfrak{q}_x}\right)\right]\mathop{\mathrm{d}r}\ (\text{by}\  (\ref{a4.62}))\\
=&\int_0^D \left[\int_{B_r(x)} \Delta f\mathop{\mathrm{d}\mathscr{H}^{n}} \frac{\partial}{\partial r}\left(\int_{Q_x}\varphi(\cdot,r)\mathop{\mathrm{d}\mathfrak{q}_x}\right)\right]\mathop{\mathrm{d}r}\ (\text{by}\  (\ref{a4.63}))\\
=&-\int_0^D  \left[\int_{\partial B_r(x)} \Delta f\mathop{\mathrm{d}\mathscr{H}^{n-1}}\int_{Q_x}\varphi(\cdot,r)\mathop{\mathrm{d}\mathfrak{q}_x}\right]\mathop{\mathrm{d}r} \ \text{((\ref{fe4.64}) and integral by parts})\\
=&\int_X  \varphi\, R_x A_x  \Delta f\mathop{\mathrm{d}\mathscr{H}^{n}}\ (\text{Prop. \ref{prop4.23} and co-area formula}).
\end{aligned}
\] 
Since $\mathrm{Test}F(X,\mathsf{d},\mathscr{H}^n)$ is dense in $D(\Delta)$ with respect to the norm $\|\cdot\|_{H^{1,2}}+\|\Delta(\cdot)\|_{L^2}$, we deduce
\begin{align}\label{feahjioeah}
R_x A_x f\in D(\Delta) \ \ \text{and}\ \  \Delta R_x A_x f=R_x A_x \Delta f. 
\end{align}

Define $\tilde{\rho}:(x,y,t)\mapsto R_x A_x \rho(x,\cdot,t)(y)$ and $\tilde{\mathrm{h}}_t f:x\mapsto \int_X \tilde{\rho}(x,y,t)f(y)\mathop{\mathrm{d}\mathscr{H}^n(y)}$ for $f\in \mathrm{Lip}(X,\mathsf{d})$. Then (\ref{feahjioeah}) implies

\begin{equation}\label{4.34}
\frac{\partial\tilde{\rho}}{\partial t} (x,\cdot,t)=R_x A_x\left(\frac{\partial \rho}{\partial t}(x,\cdot,t)\right)=R_x A_x \Delta \rho(x,\cdot,t)=\Delta \tilde{\rho}(x,\cdot,t).
\end{equation}

To prove
\begin{equation}\label{4.32}
\left\| \tilde{\mathrm{h}}_t f-f\right\|_{L^2}\rightarrow 0\ \ \text{ as $t\rightarrow 0$},
\end{equation}
we first show
\begin{equation}\label{4.31}
 \left\|\int_X \varphi \, \tilde{\mathrm{h}}_t f \mathop{\mathrm{d}\mathscr{H}^n}-\int_X \varphi  f \mathop{\mathrm{d}\mathscr{H}^n}\right\|_{L^2}\rightarrow 0\ \ \text{as }\ t\rightarrow 0,\  \forall \varphi\in L^2.
\end{equation}

{}{Due to the stochastic completeness of heat kernel established in \cite{St94}, $\{\mathrm{h}_t\}_{t>0}$ acts on $L^\infty$ as a linear family of contraction, namely for any $\phi \in L^\infty$ it holds $\sup_{t>0}\|\mathrm{h}_t \phi\|_{L^\infty} \leq \|\phi\|_{L^\infty}$. }Thus for any $x\in X$, by Bakry-\'Emery estimate \cite[Thm. 6.1]{AGS14a}, \[
{\left\|\nabla \mathrm{h}_t \left(R_x A_x f\right)\right\|_{L^\infty}}^2\leqslant \exp(-2Kt)\left\| \mathrm{h}_t\left({|\nabla R_x A_x f|}^2\right)\right\|_{L^\infty}\leqslant \exp(-2Kt)\,\mathrm{Lip}\,f,
\] 
ensuring the uniform convergence $\mathrm{h}_t (R_x A_x f)\rightarrow R_x A_x f$ as $t\to 0$. In particular, we have 
\begin{align}\label{4.68}
(\mathrm{h}_t (R_x A_x f))(y)\rightarrow (R_x A_x f)(x)=\lim_{z\rightarrow x} (R_x A_x f)(z)=f(x),\ \text{as}\ t\to 0.
\end{align}

By applying Fubini's Thm. and the co-area formula, we obtain
\[
\begin{aligned}
&\int_X \varphi \, \tilde{\mathrm{h}}_t f \mathop{\mathrm{d}\mathscr{H}^n}=\int_X \varphi(x)\int_X  R_x A_x\rho(x,\cdot,t)(y)f(y) \mathop{\mathrm{d}\mathscr{H}^n(y)}\mathop{\mathrm{d}\mathscr{H}^n(x)}\\
=\ &\int_X \varphi(x)\int_0^D\int_{\partial B_r(x)}\left(\fint_{\partial B_r(x)} \rho(x,z,t)\mathop{\mathrm{d}\mathscr{H}^{n-1}(z)}\right)f(y) \mathop{\mathrm{d}\mathscr{H}^{n-1}(y)}\mathop{\mathrm{d}\mathscr{H}^n(x)}\\
=\ &\int_X  \varphi(x) \int_0^D\int_{\partial B_r(x)} \rho(x,z,t)(R_x A_x f)(z)\mathop{\mathrm{d}\mathscr{H}^{n-1}(z)}\mathop{\mathrm{d}\mathscr{H}^n(x)}\\
=\ &\int_X \varphi(x) (\mathrm{h}_t (R_x A_x f))(x)\mathop{\mathrm{d}\mathscr{H}^n(x)}.
\end{aligned}
\]
Then (\ref{4.31}) follows from these calculations, together with the observation \[
\sup_{t>0}\int_X (\mathrm{h}_t (R_x A_x f))^2(x)\mathop{\mathrm{d}\mathscr{H}^n(x)}\leqslant \mathscr{H}^n(X)(\sup_X f)^2<\infty,
\] and the application of (\ref{4.68}) and the dominated convergence theorem. As a result, (\ref{4.32}) is derived by combining (\ref{4.31}) with the following limit.
\[
\begin{aligned}
&\int_X (\tilde{\mathrm{h}}_t f)^2 \mathop{\mathrm{d}\mathscr{H}^n}=\int_X \tilde{\mathrm{h}}_t f(x) (\mathrm{h}_t (R_x A_x f))(x)\mathop{\mathrm{d}\mathscr{H}^n(x)}\\
=\ &\int_X  (\mathrm{h}_t (R_x A_x f))^2(x)\mathop{\mathrm{d}\mathscr{H}^n(x)}\rightarrow \int_X f^2\mathop{\mathrm{d}\mathscr{H}^n},\ \ \text{as}\ t\rightarrow 0.
\end{aligned}
\]

From (\ref{4.34}), (\ref{4.32}) and the uniqueness of the solution to the heat equation, we know for any $s,t>0$, since $\mathrm{h}_s f$ is a test function, it follows $\tilde{\mathrm{h}}_t\mathrm{h}_s f=\mathrm{h}_{t+s}f$. By letting $s\rightarrow 0$ and using the dominated convergence theorem again, we obtain $\tilde{\mathrm{h}}_t f=\mathrm{h}_{t}f$. Finally, since the set of test functions is dense in $H^{1,2}$ (and hence in $L^2$), we conclude that
\[
\rho(x,y,t)=\tilde{\rho}(x,y,t),\ \forall x,y\in X,\ \forall t>0.
\] 
This verifies the strong harmonicity of $(X,\mathsf{d},\mathscr{H}^n)$, which, by Cor. \ref{cor1.5}, implies the smoothness.

\end{proof}
\section{Appendix: eigenvalue estimates for symmetric matrices}
Let $A=(a_{ij})$ be an $n\times n$ real symmetric matrix and $\lambda_{\max}$, $\lambda_{\min}$ be its maximum and minimum eigenvalues respectively.

\begin{prop}\label{prop5.1}
	For every $t\in\mathbb{R}$, it holds that 
	\[
	\lambda_{\max}\leqslant t+\sum_{i,j}|a_{ij}-t\delta_{ij}|.
	\]
\end{prop}
\begin{proof}
	Define a new matrix $B=(b_{ij})$ by $b_{ij}=a_{ij}-t\delta_{ij}$. Then $B$ is also symmetric. Let $\mu_{\max}$ be the maximum eigenvalue of $B$. Then it is obvious that $\lambda_{\max}=\mu_{\max}+t$.
	
	Since the maximum eigenvalue satisfies $\mu_{\max}=\max_{\|x\|=1}xBx^T$, for any unit vector $x=(x_1,\ldots,x_n)$ we have
	\[
	|xBx^T|=\left|\sum_{i,j}b_{ij}x_ix_j\right|\leqslant\sum_{i,j}|b_{ij}||x_ix_j|\leqslant\sum_{i,j}|b_{ij}|\frac{{x_i}^2+{x_j}^2}{2}\leqslant\sum_{i,j}|b_{ij}|.
	\]
	
Thus we conclude by the following inequality.
	\[
	\lambda_{\max}-t=\mu_{\max}\leqslant|\mu_{\max}|\leqslant\sum_{i,j}|b_{ij}|=\sum_{i,j}|a_{ij}-t\delta_{ij}|.
	\]
\end{proof}
\begin{cor}\label{cor5.2}
	Assume $A$ is positive semi-definite. The following estimates for $\lambda_{\mathrm{max}}$ and $\lambda_{\mathrm{min}}$ holds.
\begin{align}\label{5.1}
\frac{\sum_i a_{ii}}{n}\leqslant\lambda_{\mathrm{max}}\leqslant\inf_{t\in\mathbb{R}}\left(t+\sum_{i,j}|a_{ij}-t|\right),
\end{align}
\begin{align}\label{5.2}
	\mathrm{det}(A)(\lambda_{\mathrm{max}})^{1-n}\leqslant \lambda_{\min}\leqslant \frac{1}{n}\sum_{i=1}^n {a_{ii}}.
\end{align}

\end{cor}
\begin{proof}
	Let $0<\lambda_1=\lambda_{\min}\leqslant \lambda_2\leqslant\cdots\leqslant\lambda_{\max}=\lambda_n$ be all the eigenvalues of $A$. The first inequality in (\ref{5.1}) follows from the fact that 
	\[
	\sum_{i=1}^n a_{ii}=\sum_{i=1}^n\lambda_i\leqslant n\lambda_{\mathrm{max}}.
	\]
	
	As for the first inequality in (\ref{5.2}), it suffices to observe that
	\[
	\mathrm{det}(A)=\prod_{i=1}^n \lambda_i\leqslant \lambda_{\mathrm{min}}(\lambda_{\mathrm{max}})^{n-1}.
	\]
\end{proof}
\bibliographystyle{alpha}
\bibliography{reff}

\bigskip
\end{document}